\documentclass[reqno]{amsart}
\usepackage[lite]{amsrefs}
\usepackage{mathtools,amsmath,amssymb}
\usepackage[pdfencoding=auto, psdextra]{hyperref}
\hypersetup{colorlinks=true,allcolors=[rgb]{0,0,0.6}}
\usepackage[hmargin= 1.3 in,vmargin=1 in]{geometry}
\usepackage{bbm,mathrsfs}
\usepackage{xcolor}
\usepackage[all,cmtip]{xy}
\usepackage{appendix}

\newcommand{\T}{\mathbb{T}}

\newcommand{\R}{\mathbb{R}}
\newcommand{\C}{\mathbb{C}}
\newcommand{\Z}{\mathbb{Z}}
\newcommand{\N}{\mathbb{N}}

\newcommand{\eps}{\varepsilon}

\newcommand{\dd}{\mathrm{d}}

\newcommand{\e}{\mathrm{e}}

\newcommand{\mc}{\mathcal}

\newcommand{\ii}{\mathrm{i}}

\numberwithin{equation}{section}

\theoremstyle{plain}

\theoremstyle{definition}

\theoremstyle{remark}

\usepackage{graphicx, color}

\definecolor{darkred}{rgb}{0.9,0,0.3}
\definecolor{darkblue}{rgb}{0,0.3,0.9}

\newcommand{\vertiii}[1]{{\left\vert\kern-0.25ex\left\vert\kern-0.25ex\left\vert #1 
		\right\vert\kern-0.25ex\right\vert\kern-0.25ex\right\vert}}

\theoremstyle{plain} %plain, definition, remark
\newtheorem{theorem}{Theorem}[section]
\newtheorem*{theorem*}{Theorem}
\newtheorem{lemma}[theorem]{Lemma}
\newtheorem*{lemma*}{Lemma}
\newtheorem{corollary}[theorem]{Corollary}
\newtheorem*{corollary*}{Corollary}
\newtheorem{proposition}[theorem]{Proposition}
\newtheorem*{proposition*}{Proposition}

\newtheorem*{conjecture*}{Conjecture}

\theoremstyle{definition} %plain, definition, remark
\newtheorem{definition}[theorem]{Definition}
\newtheorem*{definition*}{Definition}

\newtheorem*{example*}{Example}
\newtheorem{remark}[theorem]{Remark}
\newtheorem*{remark*}{Remark}

\newtheorem*{assumption*}{Assumption}

\begin{document}
\title[Higher order Hartree equations]{The well-posedness and convergence of higher-order Hartree equations in critical Sobolev spaces on $\T^3$}
\author{Ryan L.~Acosta Babb}
\address{University of Warwick, Mathematics Institute, Zeeman Building, Coventry CV4 7AL, UK}
\email{ryan.l.acosta-babb@warwick.ac.uk}
\author{Andrew Rout}
\address{Univ Rennes, [UR1], CNRS, IRMAR -- UMR 6625, F-35000 Rennes, France}
\email{andrew.rout@univ-rennes.fr}
\begin{abstract}
In this article, we consider Hartree equations generalised to $2p+1$ order nonlinearities. These equations arise in the study of the mean-field limits of Bose gases with $p$--body interactions. We study their well-posedness properties in $H^{s_c}(\T^3)$, where $\T^3$ is the three dimensional torus and $s_c = 3/2 - 1/p$ is the scaling-critical regularity. The convergence of solutions of the Hartree equation to solutions of the nonlinear Schr\"odinger equation is proved. We also consider the case of mixed nonlinearities, proving local well-posedness in $s_c$ by considering the problem as a perturbation of the higher-order Hartree equation. In the particular case of the (defocusing) quintic-cubic Hartree equation, we also prove global well-posedness for all initial conditions in $H^1(\T^3)$. This is done by viewing it as a perturbation of the local quintic NLS.
\end{abstract}
\maketitle

\section{Introduction}
The cubic {\it Hartree equation} is given by
\begin{equation}
\label{cubic_Hartree}
\left\{
	\begin{alignedat}{2}
        \ii \partial_t u + \Delta u &= (\omega*|u|^2)u, \\
        u(x,t) &= u_0(x),
\end{alignedat}\right.
\end{equation}
where $u \colon \mathbf{X} \times I \to \C$, $\mathbf{X} = \R^d$ or $\T^d$ and $I \subset \R$, and $\omega \in L^1(\mathbf{X})$ is an even, real-valued interaction potential.
Here $*$ denotes convolution in the spatial variable, so $$f*g(x) := \int_{\mathbf{X}}f(x-y)g(y) \, \dd y.$$
The Hartree equation arises in the study of {\it mean-field} limits of the many-body Schr\"odinger equation.
Formally taking $\omega = \pm \delta$, one recovers the cubic {\it nonlinear Schr\"odinger equation}, given by
\begin{equation}
\label{cubic_NLS}
\left\{
\begin{alignedat}{2}
    \ii \partial_t u + \Delta u &= \pm |u|^2u, \\
    u(x,t) &= u_0(x),
\end{alignedat}\right.
\end{equation}
which arises in the study of Bose--Einstein condensates.
We call the case where $\omega=\delta$ the {\it defocusing} problem, and $\omega= -\delta$ the {\it focusing} problem.
In this paper we study the generalisation of \eqref{cubic_Hartree} to higher-order nonlinearities. Namely, we consider
\begin{equation}
\label{higher_order_Hartree}
\left\{
\begin{alignedat}{2}
\ii \partial_t u + \Delta u &=
    \left(\int_{\mathbf{X}^p} V(x-x_1,\ldots,x - x_p) |u(x_1)|^2 \ldots |u(x_p)|^2 \, \dd x_1 \ldots \dd x_p \right)u(x), \\
    u(x,t) &= u_0(x),
\end{alignedat}\right.
\end{equation}
where $V \in L^1(\mathbf{X}^p;\R)$ is a symmetric, real-valued interaction potential, and $V(x-x_1,\ldots,x - x_p)$ is symmetric under permutation of its arguments. Formally setting $V= \pm \delta_{\mathbf{X}^p}$, we recover the nonlinear Schr\"odinger equation (NLS) with nonlinearity $|u|^{2p}u$. Hence, we recover
\begin{equation}
\label{higher_order_NLS}
\left\{
\begin{alignedat}{2}
    \ii \partial_t u + \Delta u &= \pm |u|^{2p}u, \\
    u(x,t) &= u_0(x).
\end{alignedat}\right.
\end{equation}
In this paper, our spatial domain is the three dimensional torus, so $\mathbf{X} = \T^3$, which is the most physically relevant periodic domain. We prove that the Hartree equation given in  \eqref{higher_order_Hartree} is globally well-posed for small initial data. We also show that the solution of \eqref{higher_order_Hartree} converges to the solution of \eqref{higher_order_NLS} as $V_N \to \pm \delta$.

We take our initial condition $u_0 \in H^{s_c}(\T^3)$, where 
\begin{equation}
\label{critical_sobolev}
s_c := \frac{3}{2}- \frac{1}{p}.
\end{equation}
This exponent is critical in the Euclidean case, since it is the unique exponent
for which the scaling symmetry \[
    u(t,x) \mapsto u_\lambda(t,x) := \lambda^{-1/p}u(\lambda^{-2}t, \lambda^{-1}x), \quad \lambda>0
\] is invariant in the $\dot{H}^{s_c}_x(\R^3)$ norm.

There are two invariant quantities associated to the Hartree equation:
\begin{equation}
\label{mass}
M(u) := \|u\|_{L^2}^2 = \int_{\mathbf{X}} |u(x)|^2 \, \dd x,
\end{equation}
and
\begin{multline}
\label{energy_one_nonlinearity}
E(u) := \int_{\mathbf{X}} |\nabla u(x)|^2 \, \dd x \\
+\frac{1}{p+1} \int_{\mathbf{X}^{p+1}} V(x-x_1,\ldots,x-x_p) |u(x_1)|^2\ldots|u(x_p)|^2|u(x)|^2 \, \dd x_1 \ldots \dd x_p \, \dd x, 
\end{multline}
which are the mass and the energy (or Hamiltonian), respectively.
\begin{remark}
Let us remark that not all choices of $V \in L^1(\mathbf{X}^p)$ give rise to a Hamiltonian structure. For example, choosing $V(x,y) = \omega(x)\omega(y)$ for an even function $\omega \colon \mathbf{X} \to \R$ does not. Instead we need $V(x-x_1,\ldots,x-x_p)$ to be symmetric under permutation of its arguments $x,x_1\ldots,x_p$ and that $V(x_1,\ldots,x_p)$ be symmetric. To ensure the energy is finite, for $p=2$, we will take $V$ to be of the form
\begin{equation}
\label{three_body_interaction}
V(x-y,x-z) := \frac{1}{3}\left(\omega(x-y)\omega(y-z) + \omega(y-z)\omega(z-x) + \omega(z-x)\omega(x-y)\right),
\end{equation}
where $\omega \colon \mathbf{X} \to \R$ is an even $L^1$ function. We do this to ensure we have a convolution structure, which is convenient for the application of Young's inequality. This can be generalised to higher-order nonlinearities, see \eqref{generalised_three_body_interaction}. Interactions of this kind were considered, for example, in \cite{Lee20,NS20}. One can also consider other potentials $V$; for example, see the alternative potential discussed in Appendix \ref{appendix_hamiltonian}. In the case of $V$ as in \eqref{three_body_interaction}, the energy is of the form
\begin{equation}
\label{quintc_energy}
E(u) = \int_{\mathbf{X}} |\nabla u(x)|^2\, \dd x + \frac{1}{3} \int_{\mathbf{X}} (\omega*|u|^2)^2|u(x)|^2 \, \dd x.
\end{equation}
The finiteness of \eqref{quintc_energy} is proved in Lemma \ref{energy_finite_lemma}.
\end{remark}
\begin{remark}
Let us remark on the assumptions on the interaction potential $V$. In many-body quantum mechanics, it is standard to assume that the interaction potential $V(x,x_1,\ldots,x_p)$ is symmetric under permutation of its arguments. We note that this assumption implies that $V'$ defined via $V'(x-x_1,\ldots,x-x_p) := V(x,x_1,\ldots,x_p)$ is symmetric under permutation of its arguments $x,x_1\ldots,x_p$ and that $V'(x_1,\ldots,x_p)$ be symmetric. We slightly abuse notation in this paper and write $V'=V$ throughout. We also note that in the two-body case, this implies that the interaction potential is even. 
\end{remark}

The well-posedness of \eqref{higher_order_NLS} has been widely studied in critical Sobolev spaces, for example  in \cite{CW90,CKSTT08} on $\mathbb{R}^3$ and in \cites{GOW14,HTT11,IP11,IP12b,Lee19,Wang13} and the references therein for periodic settings. The well-posedness in $\mathbb{R}^3$ of \eqref{cubic_Hartree} was studied in \cite{GV80}. Like \eqref{cubic_Hartree}, one can also realise \eqref{higher_order_Hartree} as a mean-field limit of a many-body Schr\"odinger equation. Namely, consider
\begin{equation*}
\left\{
	\begin{alignedat}{2}
		\ii \partial_t \Psi_N &= H_N \Psi_N, \\
        \Psi_N(x,0) &\sim u_{0}^{\otimes N},
	\end{alignedat}\right.
\end{equation*}
where
\begin{equation}
\label{quantum_hamiltonian}
H_N := \sum_{j=1}^N -\Delta_j + \frac{1}{N^{p}} \sum^N_{1 \leq i_0 < i_1 < \ldots < i_{p} \leq N} V(x_{i_0}-x_{i_1},\ldots,x_{i_0} - x_{i_p}).
\end{equation}
We note that each $x_i -x_j$ can be written as the difference of $x_{i_0} - x_j$ and $x_{i_0} - x_i$. Suppose that $V \in L^1(\mathbf{X}^p)$ is symmetric, with $V(x-x_1,\ldots,x-x_p)$ symmetric under permutation of its arguments $x,x_1\ldots,x_p$. Then 
\begin{equation}
\label{many_body_convergence}
``\Psi_{N} \to u"
\end{equation}
as $N \to \infty$, where $u$ is the solution to \eqref{higher_order_Hartree} with initial condition $u_0$.
The convergence in \eqref{many_body_convergence} is in terms of reduced density matrices. %{\color{red} [REF?]}
Results of this kind are known as {\it propagation of chaos} in the mathematical physics community. 
For the cubic Hartree equation, early results of this kind were proved in \cite{Hep74, Spo80}.
We direct the reader to \cite{BPS16} and the references therein for an overview of the topic.
In the case of a three-body interaction, the problem was studied in \cite{CP11, CH19, NS20}.
For mixed $p$-body interactions, the problem in $\R^{d}$ ($d=1,2$) with a potential converging to the delta function was studied in \cite{Xie15}. Higher-order Gross-Pitaevskii hierarchies were also considered in \cite{HTX16,CP10}. In recent years, the three-body interaction, which corresponds to $p=2$ in \eqref{quantum_hamiltonian}, has received much attention in the mathematical physics community. We direct the reader, for example, to \cite{ JV24, Lee20, NS20, NRT23, RS23} and the references within. In this paper, we only study the problem from the classical perspective.

\section{Main results}
\subsection{Statement of results}
We now state our main results.

\begin{proposition}[Local well-posedness of \eqref{higher_order_Hartree}]
\label{LWP_theorem}
Fix an integer $p \geq 2$. Suppose that $u_0 \in H^{s_c}(\mathbf{X})$ with $s_c$ as in \eqref{critical_sobolev} and $V \in L^1(\mathbf{X}^{p})$.
Then there exist a $T > 0$ --- depending on $\|u\|_{H^{s_c}}$, $\|V\|_{L^1}$, and $p$ ---
and a unique $u \in C_t((-T,T);H^{s_c}(\mathbf{X})) \cap X^{s_c}((-T,T))$ such that $u$ solves
\begin{equation*}
\left\{\begin{alignedat}{2}
\mathrm{i} \partial_t u + \Delta u &= \left(\int_{\mathbf{X}^p} V(x-x_1,\ldots,x - x_p) |u(x_1)|^2 \ldots |u(x_p)|^2 \, \dd x_1 \ldots \dd x_p \right)u(x), \\
u(x,0) &= u_0(x).
\end{alignedat}\right.
\end{equation*}
\end{proposition}
In the statement of Proposition \ref{LWP_theorem}, $X^{s}$ denotes the {\it atomic space} --- see Definition \ref{X^s_definition} below. 

We also prove the following global well-posedness theorem for small initial data.

\begin{proposition}[Small data global well-posedness of \eqref{higher_order_Hartree} for $p=2$]
\label{GWP_theorem}
Suppose that $p = 2$, let $s_c = 1$, and $V$ is as in \eqref{three_body_interaction} above.
Then there is an $\eps_0 > 0$ such that for every $u_0$ with $\|u_0\|_{H^1(\T^3)} < \eps_0$ and 
every $T > 0$, there is a unique $u \in C_t((-T,T);H^{s_c}(\mathbf{X})) \cap X^{s_c}((-T,T))$ such that $u$ solves 
\eqref{higher_order_Hartree}.
\end{proposition}

\begin{theorem}[Convergence of the Hartree equation to the NLS]
\label{Convergence_theorem}
Fix an integer $p \geq 2$. Let $u_0 \in H^{s_c}$ with $s_c$ as in \eqref{critical_sobolev} and suppose that $V_N \to \delta$ as $N \to \infty$ with $V_N \geq 0$ and $\int V_N = 1$ for all $N$. Consider the equations given by
\begin{equation}
\label{p_Hartree}
\left\{
\begin{alignedat}{2}
    \mathrm{i} \partial_t u^N + \Delta u^N &= \pm \bigg(\int_{\mathbf{X}^p} V_N(x-x_1,\ldots,x - x_p) |u^N(x_1)|^2 \ldots\\
        &\qquad\qquad\qquad\qquad\ldots|u^N(x_p)|^2 \times \dd x_1 \ldots \dd x_p \bigg)u^N(x),\\
u^N(x,0) &= u_0(x),
\end{alignedat}\right.
\end{equation}
and
\begin{equation}
\label{p_NLS}
\left\{
\begin{alignedat}{2}
\mathrm{i} \partial_t u + \Delta u &= \pm |u|^{2p}u, \\
u(x,0) &= u_0(x).
\end{alignedat}\right.
\end{equation}
Let $T$ be less than the minimum of the times of existence of the solutions to \eqref{p_Hartree} and \eqref{p_NLS}. Then
\begin{equation}
\label{convergence_result}
\lim_{N \to \infty} \|u - u^N\|_{L^\infty_{[-T,T]}H^{s_c}(\T^3)} = 0.
\end{equation}
\end{theorem}
We make the following remarks about our results.
\begin{remark}
We note that we can define the time of existence $T > 0$ in Theorem \ref{Convergence_theorem} because the time of existence in Proposition \ref{LWP_theorem} depends on $\|V\|_{L^1}$, which is uniform in $N$ in Theorem \ref{Convergence_theorem}. In the particular case where $p=2$, equations \eqref{p_Hartree} and \eqref{p_NLS} are globally well-posed in the defocusing case for $N$ sufficiently large. The global well-posedness of \eqref{p_Hartree} for a specific sequence of $V_N$ for $p=2$ is proved in \cite[Appendix A]{NS20} by viewing it as a small perturbation of the quintic NLS, which was shown to be well-posed in \cite{IP11}. The convergence to the local NLS is also proved, and the proofs of both results easily generalises to any positive sequence of potentials $V_N \to \delta$.
\end{remark}
\begin{remark}
Let us remark that our results would also hold for an irrational torus. Indeed, our proofs are based on Strichartz estimates, first proved in the periodic setting in \cite{Bou93}. These were extended to irrational tori in \cite{BD15}; see also \cite{KV16}. For simplicity of presentation, we restrict ourselves to the case of rational tori.
\end{remark}
\begin{remark}
We note that the assumption $V\in L^1(\mathbf{X}^{3p})$ is not a restrictive assumption. Indeed, if one takes $V$ to be of the form of \eqref{three_body_interaction}, changing variables we find that $$V(x,y) = \frac{1}{3}(\omega(x)\omega(y) + \omega(x)\omega(x-y) + \omega(y)\omega(x-y).$$ Applying H\"older's inequality, one finds that the Coulomb potential $\omega(x) = \frac{1}{|x|}$ is a permitted interaction potential in three dimensions. We recall that because we work on a torus, we only need to account for integrability around the origin.
\end{remark}
\begin{remark}
We note that the convergence of the Hartree equation to the NLS is of particular interest from the view of mathematical physics. Indeed, the Hartree equation is often a convenient intermediate step between the quantum mechanical problem and the NLS. For example, in \cite{NS20} for the case $p=2$, the authors consider a Hamiltonian of the form
\begin{equation*}
H_N := \sum_{j=1}^N -\Delta_j + \frac{1}{N^2} \sum_{1 \leq i < j < k \leq N} N^{6\beta} V(N^\beta(x_i-x_j),N^{\beta}(x_j-x_k)),
\end{equation*} 
and then proceed via an intermediate Hartree equation --- see \cite[Appendix A]{NS20} for the problem on $\T^3$. Here we note that the Hartree equation corresponds to the effective equation for $\beta = 0$, so corresponds to a mean-field limit. The convergence of the Hartree equation to the local NLS is then used to obtain convergence of the many-body problem to the local NLS. 

It is however, as far as the authors are aware, not possible to obtain an equation which includes the scattering length using this kind of method. Indeed, if one goes via the Hartree equation, the best one can obtain is the  coupling constant (strength of the nonlinearity) $\int V(x) \, \dd x$. We direct the reader to \cite[Section 2]{NS20} and the references therein for a more detailed discussion. 
\end{remark}

\begin{remark}
We remark that one can also generalise the notion of the Hartree equation by considering more singular convolution potentials.
Indeed, in the mathematical physics literature, equations of this form are also called 
(generalised) Hartree equations.
Equations with more singular convolution potentials were considered in the case $\mathbf{X} = \R^d$ 
in \cite{AR20,ARR22,GX23} and the references within.

For instance, if we take $V$ of the form $|x|^{-b}$ with $b > d$
(so that $V$ just misses being in $L^1$), then the scaling argument in \cite{AR20} shows
that the critical Sobolev space exponent becomes \[
    s_c = \frac{d}{2} - \frac{d-b+2}{2p},
\] which is larger than the critical regularity in (\ref{critical_sobolev}) --- there for $d=3$.
(Note that our $p$ is $p-1$ in \cite{AR20}.)
Thus, as one might expect, more singular potentials come at a cost in regularity.
We leave consideration of these cases to future work.

\end{remark}
\subsection{Hartree equation with mixed nonlinearity}

We also consider the case of mixed Hartree-type nonlinearities. For $p \in \N$, write
\begin{equation}
\label{mathcal_N}
\mathcal{N}_p(u) := \left(\int_{\mathbf{X}^p} V(x-x_1,\ldots,x - x_p) |u(x_1)|^2 \ldots |u(x_p)|^2 \, \dd x_1 \ldots \dd x_p \right)u(x).
\end{equation}
We consider the nonlocal NLS given by
\begin{equation}
\label{mixed_nonlinearity}
\ii \partial_t u + \Delta u = \lambda_1 \mathcal{N}_{p_1}(u) + \lambda_2 \mathcal{N}_{p_2}(u),
\end{equation}
where $p_1 < p_2$ and $p_i \in \N$ and $\lambda_i \in \R$.
The $\lambda_i$ can be interpreted as coupling constants, determining the strength of the nonlinearities.
We argue that when considering nonlinearities of this form, we are able to treat \eqref{mixed_nonlinearity} as a perturbation of the Hartree equation with nonlinearity
$\mathcal{N}_{p_2}(u)$, so we can consider the local well-posedness in $H^{s_c(p_2)}$.
Moreover, when $p_1 = 1$, $p_2 = 2$, and $V_{p=2} \approx \delta_{2}$,
we can treat \eqref{mixed_nonlinearity} as a perturbation of the defocusing (local) quintic NLS.
In this way, we can extract global well-posedness properties for \eqref{mixed_nonlinearity} in 
$H^1$ without restriction on the size of the initial data.
Here, $\delta_p$ denotes the $\delta$ distribution in $p$ variables.
This idea comes from \cite{LYYZ23}, where the authors consider a local cubic-quintic nonlinear Schr\"odinger equation, which formally corresponds to taking
\begin{equation*}
    V_{p_1=1} = \delta_{1}, \quad V_{p_2=2} = \delta_{2}.
\end{equation*}
We have the following local well-posedness result.
\begin{proposition}
\label{mixed_nonlinearity_LWP}
Suppose that $p_2 > p_1$ and $s_c = s_c(p_2)$ is as in \eqref{critical_sobolev} with $p=p_2$. Suppose that $V_{p_i} \in L^1(\mathbf{X}^{p_i})$ and $\lambda_i \in \R \setminus \{0\}$.
Moreover, suppose $u_0 \in H^{s_c}$. Then there exist a $T > 0$ --- depending on $\|u\|_{H^{s_c}}$, $\|V\|_{L^1}$, and $p$ ---
and a unique $u \in C_t((-T,T);H^{s_c}(\mathbf{X})) \cap X^{s_c}((-T,T))$ which solves \eqref{mixed_nonlinearity}.
\end{proposition}
We have the corresponding global well-posedness result for small initial data.
\begin{proposition}
\label{mixed_nonlinearity_GWP}
Suppose that $p_1 = 1$, $p_2 = 2$, and let $s_c = 1$. Suppose further that $\lambda_i \in \R$ with $V_1,V_2$ as above corresponding to even $\omega,\omega' \in L^1(\mathbf{X})$.
Then there is $\eps_0 > 0$ such that for every $u_0$ with $\|u_0\|_{H^1(\T^3)} < \eps_0$ and every $T > 0$, there is a unique $u \in C_t((-T,T);H^{s_c}(\mathbf{X})) \cap X^{s_c}((-T,T))$ such that $u$ solves \eqref{mixed_nonlinearity}.
\end{proposition}

We prove the following global well-posedness result, which does not require a smallness of initial condition assumption. It does, however, require that we are a perturbation of the quintic NLS. The result is similar to the statement of \cite[Theorem 1.2]{LYYZ23}, which proves the corresponding result for local interactions. 

In what follows, we will use $\mathcal{N}_{2,N}$ to denote a nonlinearity of the form of \eqref{mathcal_N} defined with a potential $V_{2,N}$. Recall that we denote by $\delta_2$ the delta distribution on functions from $\mathbf{X}^2$.

\begin{theorem}
\label{GWP_mixed_nonlinearity_close_quintic}
Suppose that $V_{2,N} \to \delta_2$ with respect to continuous functions with $V_{2,N}$ as in \eqref{three_body_interaction}, and let $u_0 \in H^1(\mathbf{X})$. Let $\omega' \in L^1(\mathbf{X})$ and $\lambda \in \R$. Moreover, suppose that we have the estimate
\begin{equation}
\label{H1_estimate}
\|\nabla u\|^2_{L^\infty_t L^2_x} \lesssim M(u) + E(u).
\end{equation}
Then, for $N$ sufficiently large, the Cauchy problem given by
\begin{align}
\label{mixed_nonlinearity_equation_GWP_equation}
\left\{
\begin{alignedat}{2}
\ii \partial_t u + \Delta u &= \lambda \left(\int |u(y)|^2\omega'(x-y) \, \dd y \right) u(x)\\
    &\qquad+ \left(\int V_{2,N}(x-y,x-z)|u(y)|^2|u(z)|^2 \,\dd y \, \dd z \right) u(x), \\
u_0 &\in H^1(\mathbf{X})
\end{alignedat}
\right.
\end{align}
is globally well-posed.
\end{theorem}
\begin{remark}
We remark that we can interpret Theorem \ref{GWP_mixed_nonlinearity_close_quintic} as saying that adding an arbitrary small quintic non-linearity allows us to ensure the global well-posedness of \eqref{mixed_nonlinearity}. Indeed, we can scale the equation accordingly to absorb to get the same result for $\lambda_2 V_{2,N}$ for any $\lambda_2 > 0$.
\end{remark}
\begin{remark}
Let us comment that \eqref{H1_estimate} gives us a bound on the kinetic energy of the solution and allows  us to build the global well-posedness. Remark \ref{H1_estimate_remark} explores when \eqref{H1_estimate} holds, but it can be heuristically be understood as a condition which ensures the two interaction potentials $V_1,V_2$ can be compared. It is the nonlocal analogue of the kinetic energy bound \cite[Lemma 2.10]{LYYZ23}, which holds for all $\lambda_1 \in \R$ and $\lambda_2 > 0$ due to the local interactions.
\end{remark}
Finally, we prove the following analogue of Theorem \ref{Convergence_theorem} for \eqref{mixed_nonlinearity}.
\begin{corollary}
\label{mixed_convergence_theorem}
Suppose that $p_2 > p_1$ with $p_i \in \N$. Moreover let $s_c = s_c(p_2)$ is as in \eqref{critical_sobolev} with $p=p_2$. Suppose further that $u_0 \in H^{s_c}$, $V_{p_i,N} \to \delta_{p_i}$ with $\|V_{p_i,N}\|_{L^1(X^{p_i})} = 1$ for $i=1,2$, and $V_{p_i,N} \geq 0$. Consider the equations given by
\begin{equation*}
\left\{\begin{alignedat}{2}
\ii \partial_t u^N + \Delta u^N &= \pm \mc{N}_{p_1,N}(u^N) \pm \mc{N}_{p_2,N}(u^N), \\
u^N(x,0) &= u_0(x)
\end{alignedat}\right.
\end{equation*}
and
\begin{equation}
\label{mixed_local}
\left\{\begin{alignedat}{2}
\ii \partial_t u + \Delta u &= \pm |u|^{2p_1}u \pm |u|^{2p_2}u, \\
u(x,0) &= u_0(x).
\end{alignedat}\right.
\end{equation}
Suppose that $T^N$ and $T$ are the maximum times of existence of $u^N$ and $u$ respectively and suppose $T'< \min_N\{T_N,T\}$. Then
\begin{equation}
\label{mixed_nonlinearity_convergence_result}
\lim_{N \to \infty} \|u - u^N\|_{L^\infty_{[-T',T']}H^{s_c}(\T^3)} = 0.
\end{equation}
\end{corollary}
\begin{remark}
Let us remark that the maximum time of existence of $u^N$ does not shrink to zero as $N \to \infty$
by the assumption that $\|V_{p_2,N}\|_{L^1} =1$ for all $N$.
\end{remark}

\subsection{Previously known results}
\label{previously_known_results_section}
The well-posedness of \eqref{cubic_Hartree} in $H^1(\mathbb{R}^3)$ was considered in \cite{GV80}. The well-posedness of \eqref{higher_order_NLS} was considered in subcritical spaces in \cite{Kat87}, and in critical Sobolev spaces by \cite{CW90}, before global well-posedness was extended to large initial data with critical regularity in \cite{CKSTT08}. See also \cite{TVZ07} for the well-posedness of the NLS with mixed power-type nonlinearities in $\R^d$.

We now mention the known results in the periodic setting. The global well-posedness of the local quintic NLS for the energy-critical space for small initial data was established in \cite{HTT11}, and for all initial data in \cite{IP11}. The local and global well-posedness for small initial data (depending on the exponent and dimension of the space) was investigated in all dimensions and for various values of $p \in \N$ for the local NLS in \cite{Wang13}. We direct the reader to \cite[Theorem 1.3]{Wang13} for a precise statement of the result. The well-posedness of the local NLS for certain choices of $p \in \N$ (depending on the dimension) was studied on irrational tori in \cite{GOW14}, see \cite[Theorems 1.4 and 1.5]{GOW14} for a precise statement. The well-posedness of the local NLS was also studied for general $p \in \R$, $p \geq 1$ on $\T^3$ in \cite{Lee19}.

The convergence of the one dimensional cubic Hartree equation to the nonlocal NLS was proved in \cite{FKSS18} for initial data in $H^s(\T)$ for $s > \frac{2}{3}$. The convergence of the quintic Hartree equation in one spatial dimension was studied in \cite{RS23}. Convergence to the local quintic NLS was proved for the local time of existence, and for Gibbsian initial data, this result can be extended almost surely to arbitrary times of existence. The well-posedness and convergence of the quintic Hartree equation, corresponding to $p=2$, is considered in \cite[Appendix A]{NS20}. More precisely, the authors prove that, for a particular choice of sequence of $V_N$'s converging to the $\delta$ distribution, one has global well-posedness of \eqref{higher_order_Hartree} for $p=2$ if the potential $V_N$ is close enough to the $\delta$ function. This is done by considering the Hartree equation as a small perturbation of the quintic NLS. The proof generalises to the case of any sequence $V_N$ converging to the $\delta$ function. The ill-posedness of the cubic Hartree equation was studied in \cite{BS23}.

Finally, the global well-posedness of the local NLS with cubic-quintic nonlinearity was recently studied on $\T^3$ in \cite{LYYZ23}. In this paper, they proved global well-posedness for a positive quintic nonlinearity by viewing the equation as a perturbation of the quintic NLS, adapting an argument in $\R^3$ from \cite{Zhang06}.

\subsection{Outline of the paper}
To make the fixed-point argument to prove Proposition \ref{LWP_theorem}, we need to prove a multilinear estimate, which is done in Section \ref{multlinear_section}.
We prove Proposition \ref{GWP_theorem} using conservation of energy --- see Section \ref{well-posedness_subsection}.
Theorem \ref{Convergence_theorem} is proved in Section \ref{convergence_section} by applying the multilinear estimates. In Section \ref{mixed_nl_section}, we consider the case of a mixed Hartree nonlinearity, and extend the results of Proposition \ref{GWP_theorem} and Theorem \ref{Convergence_theorem} to this setting.
We also show that, for perturbations of the defocusing quintic equation, the mixed (quintic-cubic) Hartree equation is globally well-posed without smallness conditions;
see Theorem \ref{GWP_mixed_nonlinearity_close_quintic} for a precise statement.
The proof relies on multilinear estimates for the subcritical nonlinearities ---
see Propositions \ref{mixed_multlinear_proposition} and \ref{multilinearZ_estimates}.
Finally, in Appendix \ref{appendix_hamiltonian}, we consider some specific examples of interaction potentials $V_N$.
We explain how to apply the results of Theorem \ref{Convergence_theorem} in these particular examples.

\section{Well-posedness of the Hartree equation}
\label{multlinear_section}
In this section, we prove the well-posedness of \eqref{higher_order_Hartree}. To do this, we state and prove energy bounds and multilinear estimates which will form the base of our arguments.
We first introduce some notation and functions spaces.
\subsection*{Notation and preliminaries}
We write $A \lesssim B$ to denote $A \leq C B$ for some constant $C > 0$, and $A \sim B$ for $A \lesssim B \lesssim A$. Where it is clear, we will omit the domain of integration from our integrals. We define space-time $L^pL^q$ norms by
\begin{equation*}
\|u\|_{L^p([0,T])L^q(\mathbf{X})} := \left(\int^T_0 \left(\int_{\mathbf{X}} |u(x,t)|^q \, \dd x\right)^{\frac{p}{q}}\dd t\right)^{\frac{1}{p}},
\end{equation*}
where the definition is altered accordingly when $p$ or $q = \infty$. We often omit the domains from the norm, writing $L^p_tL^q_x$, and we write $L_t^pL^p_x \equiv L^p_{t,x}$.

For $k \in \Z^3$, the Fourier coefficients of $f$ are
\begin{equation*}
\hat{f}(k) := \frac{1}{(2\pi)^{\frac{3}{2}}} \int_{\mathbf{X}} f(x)\e^{-\ii k\cdot x}\dd x.
\end{equation*}

\iffalse
We denote by
\begin{equation*}
\tilde{f}(k,\tau) := \frac{1}{4\pi^2} \int_{X\times \R} f(x,t) \e^{-\ii (k \cdot x + \tau t)} \,\dd t \dd x
\end{equation*}
the space-time Fourier transform of $f$, where $\tau \in \R$.
\fi

To define a partition of unity, we fix a non-negative function $\phi \in C^\infty_0(\R)$
with $\phi(x) = 1$ for $|x| < 1$ and $\phi(x) = 0$ for $|x| > 2$. Let $M \in 2^{\N}$ be a dyadic number.
Set $\phi_1(k) := \phi(|k|)$. For $M \geq 4$, define
\begin{equation*}
\phi_M(k) := \phi\left(\frac{|k|}{M}\right) - \phi\left(\frac{2|k|}{M}\right).
\end{equation*}
For $f \in L^2$, we define the Littlewood--Paley operator by
\begin{equation*}
\widehat{P_M f}(k) := \phi_M(k) \hat{f}(k).
\end{equation*}
We also write $P_{\leq M} := \sum_{1 \leq M' \leq M} P_{M'}$. More generally, for a set $\mathcal{C} \subset \Z^3$, we will write
\begin{equation*}
\widehat{P_{\mathcal{C}} f}(k) := \chi_{\mc{C}}\hat{f}(k),
\end{equation*}
where $\chi_\mathcal{C}$ denotes the indicator function on $\mc{C}$.

For $s \in \R$, we define the Sobolev space $H^s$ as the set of functions for which
\begin{equation*}
\|f\|^2_{H^s} := \sum_{k \in \Z^3} \langle k \rangle^{2s} |\hat{f}(k)|^2 \sim \sum_{M \geq 1} M^{2s} \|P_M f\|^2_{L^2} < \infty.
\end{equation*}
Here, when we write $M \geq 1$, we mean $M \in 2^{\N}$.
\subsubsection*{Atomic spaces}
We now introduce the atomic spaces, first used in the context of dispersive PDEs in \cite{HHK09,HTT11}. Let $\mc{H}$ be a separable complex Hilbert space. Fix a time interval $I := [0,T)$. Let $\mc{Z}$ denote the set of finite partitions $0=t_0<t_1 \ldots <t_n \leq T$.
\begin{definition}
Fix $p \in [1,\infty)$. A $U^p$-atom is a function $a\colon [0,T) \to \mc{H}$ such that
\begin{equation*}
a = \sum_{j=0}^{n-1} \chi_{[t_j,t_{j+1}]} \psi_j,
\end{equation*}
where $\{t_j\} \in \mc{Z}$ and $\{\psi_k\} \subset \mc{H}$ with $\sum_{j=0}^{n-1} \|\psi_j\|_{\mc{H}}^p \leq 1$. 
\end{definition}
\begin{definition}
The space $U^p(\R;\mc{H})$ is defined as the set of functions $u \colon \R \to \mc{H}$ such that
\begin{equation*}
u = \sum_{j=1}^\infty \lambda_j a_j,
\end{equation*}
where $a_j$ are $U^p$-atoms and $\sum_j |\lambda_j| < \infty$.
We define the norm $\|\cdot\|_{U^p}$ by
\begin{equation*}
\|u\|_{U^p} := \inf \left\{ \sum_{j=1}^\infty |\lambda_j| : u = \sum_{j=1}^\infty \lambda_j a_j, \lambda_j \in \C, \text{$a_j$ are $U^p$-atoms} \right\}.
\end{equation*}
\end{definition}
\begin{definition}
We define the space $V^p(\R;\mc{H})$ to be the space of functions $v\colon\R \to \mc{H}$ such that
\begin{equation*}
\|v\|_{V^p} := \sup \left( \sum_{j=1}^n \|v(t_j) - v(t_{j-1})\|_{\mc{H}}^p \right)^{\frac{1}{p}}
\end{equation*}
is finite, where the supremum is taken over the set of finite partitions of $I$.
We denote by $V_{\text{rc}}^p$ the closed subspace of $V^p$ containing the right continuous functions $v$ such that $\lim_{t \to T} v(t) = 0$.
\end{definition}
We now set the Hilbert space $\mc{H}$ to be $H^s_x$.
\begin{definition}
\label{X^s_definition}
Fix $s \in \R$. We denote by $X^s([0,T))$ and $Y^s([0,T))$ the Banach spaces of all functions $u \colon [0,T) \to H^s(\T^3)$ such that for each $k \in \Z^3$, the map $t \mapsto \widehat{\e^{-\ii t \Delta}u(t)}(k)$ is in $U^2([0,T))$ or $V^2_{\text{rc}}([0,T))$ respectively. We define the norms
\begin{align*}
\|u\|^2_{X^s([0,T))} &:= \sum_{k \in \Z^3} \langle k\rangle^{2s}
    \left\| \widehat{\e^{-\ii t \Delta}u(t)(k)} \right\|_{U^2}^2, \\
\|u\|^2_{Y^s([0,T))} &:= \sum_{k \in \Z^3} \langle k\rangle^{2s}
    \left\| \widehat{\e^{-\ii t \Delta}u(t)(k)} \right\|_{V^2}^2.
\end{align*}
\end{definition}
From \cite{HHK09}, we have the following embeddings.
\begin{equation}
\label{embeddings}
X^s(I) \hookrightarrow Y^s(I) \hookrightarrow L^\infty_t H^s_x.
\end{equation}
We also adopt the notation that for $s \in \R$,
\begin{equation*}
\|F\|_{N^{s}(I)} := \left\|\int_I \e^{\ii \Delta (t-t')} F(t') \, \dd t' \right\|_{X^s(I)}.
\end{equation*}
We have the following duality result about $X^s$ and $Y^s$ norms, see \cite[Proposition 2.11]{HTT11}.
\begin{proposition}
\label{duality_proposition} Let $s \geq 0$ and $I$ be a bounded interval.
For $F \in L^1(I;H^s(\mathbf{X}))$, we have
\begin{equation*}
\|F\|_{N^s(I)} \leq \sup \left| \int_I \int_{\mathbf{X}} F(x,t) \overline{v(x,t)}\, \dd x \, \dd t\right|,
\end{equation*}
where the supremum is taken over $v \in Y^{-s}(I)$ with $\|v\|_{Y^{-s}(I)} \leq 1$.
\end{proposition}
From \cite[Proposition 2.11]{HTT11}, we have
\begin{equation}
\label{Ns_embedding}
\|F\|_{N^s(I)} \lesssim \|F\|_{L^1_t H^s_x(I \times \mathbf{X})}.
\end{equation}
We will use the following Strichartz estimates, see \cite{KV16}.
\begin{proposition}
\label{Strichartz_estimate_proposition}
    Suppose $q > \frac{10}{3}$ and $I$ is a bounded interval. Then
\begin{equation}
\label{Strichartz_estimate}
        \|\e^{\ii t \Delta} P_{M}u\|_{L^q_{t,x}(I\times\mathbf{X})} \lesssim M^{\frac{3}{2}-\frac{5}{q}}
        \|P_{M}u\|_{Y^0(I)}.
\end{equation}
\end{proposition}

Finally, we introduce an interpolation norm, which will be needed in the proof of Theorem \ref{GWP_mixed_nonlinearity_close_quintic}; see \cite{IP11}.
First, we need the norm
\[
    \|u\|_{Z(I)} := \sum_{p=4.1,100}\sup_{\substack{J\subset I\\|J|\leq 1}}\left(
        \sum_{M\geq 1} M^{5-\frac{p}{2}}\|P_Mu\|_{L_{t,x}^p(J)}^p\right)^{\frac{1}{p}}.
\] Note that
\[
    \|u\|_{Z(I)} \lesssim \|u\|_{X^1(I)}.
\] The interpolation norm $\|\cdot\|_{Z'(I)}$ is defined as follows: \[
    \|u\|_{Z'(I)} := \|u\|_{Z(I)}^{\frac{1}{2}}\|u\|_{X^1(X)}^\frac{1}{2}.
\]

\subsection{Energy estimates}
We collect some bounds on the energy \eqref{energy_one_nonlinearity}.
\begin{lemma}
\label{energy_finite_lemma}
Suppose $p=2$.
Moreover, let $\omega \colon \mathbf{X} \to \R$ be an even $L^1$ function
and $u \in H^1(\mathbf{X})$.
Take $V$ as in \eqref{three_body_interaction}. Then
\begin{equation}
\label{quintic_energy}
|E(u)| \lesssim \|u\|_{H^1}^2 + \|\omega\|^2_{L^1}\|u\|_{H^{1}}^{6} < \infty.
\end{equation}
\end{lemma}
\begin{proof}
Bounding the first term in \eqref{quintc_energy} is simple. For the second term, we apply Young's inequality and the Sobolev embedding theorem to get
\begin{equation*}
\left|\int (\omega*|u|^2)^2|u(x)|^2 \, \dd x\right| \leq \|\omega\|_{L^1}^{2} \|u\|_{L^{6}}^{6} \lesssim \|\omega\|_{L^1}^{2} \|u\|_{H^{1}}^{6} < \infty.
\end{equation*}
\end{proof}
\begin{remark}
\label{energy_finite_remark}
The same proof shows that if $u$ and $\omega$ are as in the statement of Lemma \ref{energy_finite_lemma}, then
\begin{equation*}
\left| \int |u(x)|^2\omega(x-y)|u(y)|^2 \, \dd x \, \dd y \right| \leq \|\omega\|_{L^1}\|u\|_{L^4}^4 \lesssim \|\omega\|_{L^1}\|u\|_{H^1}^4 < \infty.
\end{equation*}
\end{remark}

\subsection{Multilinear estimates}
We prove an estimate of the following form.
\begin{proposition}
\label{multilinear_lemma}
Fix $p \geq 2$, let $s_c$ be as in \eqref{critical_sobolev}, and suppose $I = (a,b)$ with $|I| < 1$. Let $s \geq s_c$ and suppose $u_j \in X^{s}(I)$ for $j = 1,\ldots,2p+1$. Then,
\begin{multline}
\label{multilinear_estimate}
\bigg\| \int_{\mathbf{X}^p} \bigg[ V(x-x_1,\ldots,x-x_p) \tilde{u}_{1}(x_1,t) \tilde{u}_{2}(x_1,t) \ldots \tilde{u}_{2p-1}(x_p,t) \tilde{u}_{2p}(x_p,t) \bigg] \\
\times \dd x_1 \ldots \dd x_p \, \tilde{u}_{2p+1}(x,t) \bigg\|_{N^{s_c}(I)} 
\lesssim \|V\|_{L^1} \prod^{2p+1}_{j=1} \|u_j\|_{X^{s_c}(I)}.
\end{multline}
Here $\tilde{u}_j$ denotes either $u_j$ or $\overline{u}_j$.
\end{proposition}
\begin{proof}
    To make our notation more compact and simplify the calculations,
    we will change variables in the $\mathbf{X}^p$ integral and set
    \begin{equation*}
            F_{V}(\mathbf{u})(x,t) := \int_{\mathbf{X}^p} V(\mathbf{x})
                \prod_{j=1}^{p}\tilde{u}_{2j-1}(x-x_{j},t')\tilde{u}_{2j}(x-x_{j},t') \, \dd \mathbf{x}^p
                \tilde{u}_{2p+1}(x,t)
    \end{equation*}
    for the remainder of the proof. Note that $x \in \mathbf{X}$, but $\mathbf{x} = (x_1,\ldots, x_p)\in\mathbf{X}^p$.
    We also assume, without loss of generality, that $V\geq 0$.
    
    Using Proposition \ref{duality_proposition} and the convergence of the Littlewood--Paley
    projections, it suffices to bound
    \begin{equation}\label{eqn:dualityclaim}
        \left|\int_{I\times\mathbf{X}} (P_{\leq M}F_{V}(\mathbf{u})(x,t))\overline{v(x,t)}\,\dd x \dd t\right|
            \lesssim \|P_{\leq M}v\|_{Y^{-s_c}(I)}\|V\|_{L^1}\prod_{j=1}^{2p+1}\|u_j\|_{X^{s_c}(I)},
    \end{equation} with an implicit constant that does not depend on $M$,
    and take $M\to\infty$.

    Now, let $u_0 := P_{\leq M}v$ and consider the decompositions \[
        u_j = \sum_{M_j} P_{M_j}u_j \quad \text{for all} \quad 0 \leq j \leq 2p+1,
    \] where the $M_j \in \mathcal{M}$ range over dyadic integers.
    Applying the triangle inequality in the $\dd x \,  \dd t$ integral and recalling that $V\geq 0$.
    we are left to bound
    \[
        S := \sum_{\mathcal{M}}\int_{\mathbf{X}^p}
            V(\mathbf{x})
            \left\|P_{M_0}\tilde{u}_0\prod_{j=1}^{2p+1}P_{M_{j}}\tilde{u}_j(\cdot-x_j)%
            \right\|_{L^1_{t,x}(I\times\mathbf{X})} \,\dd{\mathbf{x}^p}.
    \]
    Since $I$ remains fixed and all spatial integrals are over $\mathbf{X}$,
    we will omit these domains when writing our norms for the remaining calculations.
    
    We now consider two cases.

    {\it Case 1:} First, consider the situation in which $M_0$ is comparable to the
    highest frequency of the remaining $M_j$ (i.e. for $1\leq j \leq 2p+1$).
    Without loss of generality, we can sum over the set \[
        \mathcal{M}_1 :=\{ M_0 \sim M_{2p+1} \geq M_1 \geq \cdots \geq M_{2p}\} \cap \mathcal{M}.
    \] In order to deal with the higher frequencies $M_0$,
    we use box localisation:
    cover the scale $M_0$ with cubes $\mathcal{C}_m$ of scale $M_1$.
    Following \cite{CBU24}, we write $\mathcal{C}_m\sim\mathcal{C}_n$
    if and only if $\mathcal{C}_m+\mathcal{C}_n$
    overlaps the Fourier support of $P_{2M_1}$.

    Applying H\"older's inequality,
    using the translation invariance of the $L^q_{t,x}$ norms,
    and applying the Strichartz estimate \eqref{Strichartz_estimate}, we obtain
    \begin{align*}
        S_1 &:= \sum_{\mathcal{M}}\int_{\mathbf{X}^p}
            V(\mathbf{x})
            \left\|P_{M_0}\tilde{u}_0\prod_{j=1}^{2p+1}P_{M_{j}}\tilde{u}_j(\cdot-x_j)%
            \right\|_{L^1_{t,x}(I\times\mathbf{X})} \,\dd{\mathbf{x}^p}\\
        &\lesssim \|V\|_{L^1} \sum_{\mathcal{M}_1}\sum_{\mathcal{C}_m\sim\mathcal{C}_n}
            \|P_{\mathcal{C}_m}P_{M_0}u_0\|_{L^4_{t,x}}
            \|P_{\mathcal{C}_n}P_{M_{2p+1}}u_{2p+1}\|_{L^4_{t,x}}\\
            &\qquad\times\|P_{M_1}u_1\|_{L^4_{t,x}}\|P_{M_2}u_2\|_{L^4_{t,x}}
            \prod_{j=3}^{2p}\|P_{M_j}u_j\|_{L^\infty_{t,x}}\\
        &\lesssim\sum_{\mathcal{M}_1}\sum_{\mathcal{C}_m\sim\mathcal{C}_n}
            \left[M_1^{\frac{3}{4}} M_2^{\frac{1}{4}}\prod_{j=3}^{2p}M_{j}^{\frac{3}{2}}\right]
            \|P_{\mathcal{C}_m}P_{M_0}u_0\|_{Y^0}
            \|P_{\mathcal{C}_n}P_{M_{2p+1}}u_{2p+1}\|_{Y^0} \prod_{j=1}^{2p}\|P_{M_j}u_j\|_{Y^0}.
    \end{align*}
    Note that we have absorbed the norm $\|V\|_{L^1}$ into the implicit constant,
    since it does not depend on any other parameter.
    
    Next, for $1 \leq j \leq 2p$, we apply the Bernstein inequalities \[
        \|P_{M_j}u_j\|_{Y^0} \lesssim M_j^{-s_c}\|P_{M_j}u_j\|_{Y^{s_c}},
    \] followed by the embeddings $X^s \hookrightarrow Y^s$.
    Using (\ref{critical_sobolev}), we obtain
    \begin{equation*}
        S_1 \lesssim\sum_{\mathcal{M}_1}\sum_{\mathcal{C}_m\sim\mathcal{C}_n}
            \frac{\prod_{j=1}^{2p}M_j^\frac{1}{p}}{M_1^{\frac{3}{4}}M_2^{\frac{5}{4}}}
            \|P_{\mathcal{C}_m}P_{M_0}u_0\|_{Y^0}
            \|P_{\mathcal{C}_n}P_{M_{2p+1}}u_{2p+1}\|_{Y^0}
            \prod_{j=1}^{2p}\|P_{M_j}u_j\|_{X^{s_c}}.
    \end{equation*}

    The next step is to sum up the various projections $P_{M_j}u_j$ to obtain the norms
    $\|u_j\|_{X^s}$. To do this, we iteratively apply Cauchy--Schwarz,
    beginning with the sum over $M_{2p-1}\geq M_{2p}$:
    \begin{align*}
        \sum_{M_{2p-1}\geq M_{2p}}\frac{M_2^{\frac{1}{p}} M_{2p}^{\frac{1}{p}}}%
                {M_2^{\frac{5}{4}}} \|P_{M_{2p}}u_{2p}\|_{X^{s_c}} &\leq \frac{M_2^{\frac{2}{p}}}{M_2^{\frac{5}{4}}}
            \left(%
                \sum_{M_{2p-1}\geq M_{2p}}\frac{M_{2p}^{\frac{2}{p}}}{M_{2}^{\frac{2}{p}}}%
            \right)^{\frac{1}{2}}
            \left(%
                \sum_{M_{2p-1}\geq M_{2p}} \|P_{M_{2p}}u_{2p}\|_{X^{s_c}}^2%
            \right)^{\frac{1}{2}}\\
        &\lesssim \frac{M_2^{\frac{2}{p}}}{M_2^{\frac{5}{4}}}\|u_{2p}\|_{X^{s_c}},
    \end{align*}
    and so on until we arrive at
    \begin{multline}\label{M1M2balance_case1}
        \prod_{j=3}^{2p}\|u_j\|_{X^{s_c}}
        \sum_{M_0 \sim M_{2p+1}\geq M_1}\sum_{\mathcal{C}_m\sim \mathcal{C}_{n}}
            \frac{M_2^{\frac{3}{4}-\frac{2}{p}}}{M_1^{\frac{3}{4}-\frac{1}{p}}}
            \|P_{\mathcal{C}_m}P_{M_0}u_0\|_{Y^0}
            \|P_{\mathcal{C}_n}P_{M_{2p+1}}u_{2p+1}\|_{Y^0}\\
            \times\|P_{M_1}u_1\|_{X^{s_c}}\|P_{M_2}u_2\|_{X^{s_c}}.
    \end{multline}
    Two more applications of Cauchy--Schwarz and summing first over $M_{2p+1}\geq M_{1}$,
    then over $\mathcal{C}_m\sim\mathcal{C}_n$, leads to
    \begin{align}\label{eqn:S1}
        S_1 &\lesssim \prod_{j=1}^{2p}\|u_j\|_{X^{s_c}}\sum_{M_0\sim M_{2p+1}}\|P_{M_0}u_0\|_{Y^0}\|P_{M_{2p+1}}u_{2p+1}\|_{Y^0}\nonumber\\
            &\lesssim \prod_{j=1}^{2p}\|u_j\|_{X^{s_c}}\sum_{M_0\sim M_{2p+1}}M_0^{s_c}\|P_{M_0}u_0\|_{Y^0}M_{2p+1}^{-s_c}\|P_{M_{2p+1}}u_{2p+1}\|_{Y^0}\nonumber\\
            &\lesssim \prod_{j=1}^{2p}\|u_j\|_{X^{s_c}}
                \left(\sum_{M_0\sim M_{2p+1}}M_0^{2s_c}\|P_{M_0}u_0\|_{Y_0}^2\right)^\frac{1}{2} \left(\sum_{M_0\sim M_{2p+1}}M_{2p+1}^{-2s_c}\|P_{M_{2p+1}}u_{2p+1}\|_{Y_0}^2\right)^\frac{1}{2} \nonumber\\
        S_1 &\lesssim \|V\|_{L^1}\|u_0\|_{Y^{-s_c}}\prod_{j=1}^{2p+1}\|u_j\|_{X^{s_c}}.
    \end{align}
    In the final line, we restore the suppressed norm of $V$.

    {\it Case 2:} We next consider the situation in which the largest frequencies are $M_1\sim M_{2p+1}$,
    and so sum over the set \[
        \mathcal{M}_2 := \{M_0 \leq M_1 \sim M_{2p+1} \geq M_2 \geq \cdots \geq M_{2p}\}\cap \mathcal{M}.
    \] In this case, we don't need to use box localisation and we can apply H\"older's inequality
    directly to all the $P_{M_j}u_j$, followed by Strichartz estimates and Bernstein's inequalities
    as before:
    \begin{align*}
        S_2 &:= \sum_{\mathcal{M}}\int_{\mathbf{X}^p}
            V(\mathbf{x})
            \left\|P_{M_0}\tilde{u}_0\prod_{j=1}^{2p+1}P_{M_{j}}\tilde{u}_j(\cdot-x_j)%
            \right\|_{L^1_{t,x}(I\times\mathbf{X})} \,\dd{\mathbf{x}^p}\\
        &\lesssim \sum_{\mathcal{M}_2}
            \|P_{M_0}u_0\|_{L^4_{t,x}}
            \|P_{M_{2p+1}}u_{2p+1}\|_{L^4_{t,x}}
            \|P_{M_1}u_1\|_{L^4_{t,x}}\|P_{M_2}u_2\|_{L^4_{t,x}} \prod_{j=3}^{2p}\|P_{M_j}u_j\|_{L^\infty_{t,x}} \\
        &\lesssim\sum_{\mathcal{M}_2}\left[(M_0M_{2p+1}M_1M_2)^{\frac{1}{4}}\prod_{j=3}^{2p}M_j^\frac{3}{2}\right]
            \prod_{j=0}^{2p+1}\|P_{M_j}u_j\|_{Y^0}\\
        &\lesssim\sum_{\mathcal{M}_2}\frac{M_0^{\frac{7}{4}}\prod_{j=1}^{2p+1}M_j^{\frac{1}{p}}}%
            {M_0^{\frac{1}{p}}M_{2p+1}^{\frac{5}{4}}M_1^{\frac{5}{4}}M_2^{\frac{5}{4}}}
            \|P_{M_0}u_0\|_{Y^{-s_c}}\prod_{j=1}^{2p+1}\|P_{M_j}u_j\|_{X^{s_c}}.
    \end{align*}
    Once again, we have absorbed $\|V\|_{L^1}$ into the implicit constant.

    As in Case 1, we successively apply Cauchy--Schwarz to sum the various
    projections $P_{M_j}u_{j}$ and recover the norms of the $u_j$.
    First, we sum over $M_0 \leq M_{2p+1}$:
    \begin{equation*}
        \sum_{M_0 \leq M_{2p+1}}
            \frac{M_0^{\frac{7}{4}-\frac{1}{p}}}{M_{2p+1}^{\frac{5}{4}-\frac{1}{p}}}
            \|P_{M_0}u_0\|_{Y^{-s_c}}
        \leq M_{2p+1}^{\frac{1}{2}}\|u_0\|_{Y^{-s_c}}.
    \end{equation*}
    Next, we iteratively apply Cauchy--Schwarz and sum over $M_{j-1} \geq M_j$
    in the range $3 \leq j \leq 2p$ to obtain
    \begin{multline}\label{M1M2balance_case2}
        S_2 \lesssim \|u_0\|_{Y^{-s_c}}\prod_{j=3}^{2p}\|u_j\|_{X^{s_c}}\\
            \times\sum_{M_{2p+1}\sim M_1\geq M_2}
            \frac{M_{2p+1}^{\frac{1}{2}}}{M_1^{\frac{1}{2}}}\frac{M_2^{\frac{3}{4}-\frac{1}{p}}}{M_1^{\frac{3}{4}-\frac{1}{p}}}
            \|P_{M_{2p+1}}u_{2p+1}\|_{X^{s_c}}
            \|P_{M_{1}}u_{1}\|_{X^{s_c}}
            \|P_{M_{2}}u_{2}\|_{X^{s_c}}.
    \end{multline}
    Two more applications of Cauchy--Schwarz, summing first over
    $M_1\geq M_2$ and then over $M_1\sim M_{2p+1}$, leads us to
    \begin{equation}\label{eqn:S2}
        S_2 \lesssim \|V\|_{L^1}\|u_0\|_{Y^{-s_c}}\prod_{j=1}^{2p}\|u_j\|_{X^{s_c}},
    \end{equation}
    where we have restored the suppressed norm of $V$.
    
    Combining (\ref{eqn:S1}) and (\ref{eqn:S2}),
    and recalling that $u_0 = P_{\leq M}v$,
    we have shown (\ref{eqn:dualityclaim}),
    as required.
\end{proof}
\begin{remark}
    One can also prove the multilinear estimates (\ref{multilinear_estimate})
    by using trilinear estimates; see \cite{IP11,IP12b,LYYZ23}.
    This approach involves the interpolation norm $\|\cdot\|_{Z'}$
    and assumes $p\geq 2$.
    We follow the approach in \cite{CBU24}, which allows us to work directly
    with the $X^{s_c}$ norms.
    However, trilinear estimates are required for proof of Proposition \ref{multilinearZ_estimates} below.
\end{remark}

\subsection{Well-posedness arguments for \eqref{higher_order_Hartree}}
\label{well-posedness_subsection}
We are now able to prove Propositions \ref{LWP_theorem} and \ref{GWP_theorem}.

\begin{proof}[Proof of Proposition \ref{LWP_theorem}]
For completeness, we give the proof of local well-posedness (for small initial data). We consider the map
\begin{equation}
\label{contraction_map}
L(u) := \e^{\ii t\Delta} u_0 - \ii \int^{t}_0 \e^{\ii(t-t')\Delta} \mathcal{N}_V(u)(t') \dd t',
\end{equation}
where $$\mc{N}_V(u) :=  \left(\int_{\mathbf{X}^p} V(x-x_1,\ldots,x - x_p) |u(x_1)|^2 \ldots |u(x_p)|^2 \, \dd x_1 \ldots \dd x_p \right)u(x).$$Suppose that $\|u_0\|_{H^{s_c}(\T^3)} < \eps$ for $\eps > 0$ sufficiently small to be determined later. We show that $L$ is a contraction on the ball
\begin{equation*}
B_1 := \{u \in X^{s_c}([0,T)) \cap C_t([0,T);H^{s_c}(\T^3)) : \|u\|_{X^{s_c}([0,T))} \leq 2\eps\}
\end{equation*}
with respect to the $X^{s_c}$ norm. Proposition \ref{multilinear_lemma} implies that
\begin{equation}
\label{contraction_1}
\|L u\|_{X^{s_c}} \leq \|u_0\|_{H^{s_c}} + C\|V\|_{L^1}\|u\|_{X^{s_c}}^{2p+1}.
\end{equation}
For $\eps$ sufficiently small, it follows from \eqref{contraction_1} that $L$ maps $B_1$ to itself.
It remains to show that $L$ is a contraction on $B_1$. We have
\begin{equation*}
\|L(u-v)\|_{X^{s_c}} = \left\|\int_0^t \e^{\ii(t-t')\Delta} \left(\mc{N}_V(u) -\mc{N}_V(v) \right)\right\|_{X^{s_c}}
\end{equation*}
Using Proposition \ref{multilinear_lemma} again, we have
\begin{equation}
\label{contraction_2}
\|L(u-v)\|_{X^{s_c}} \leq C\|V\|_{L^1}\left(\|u\|_{X^{s_c}} + \|v\|_{X^{s_c}}\right)^{2p}\|u-v\|_{X^{s_c}}.
\end{equation}
Picking $\eps > 0$ sufficiently small, \eqref{contraction_2} implies that $L$ is a contraction on $B_1$. By the contraction mapping theorem, we obtain a local solution $u$.

For large data, we can argue similarly to the large data case from \cite[Theorem 1.1]{HTT11}. For future reference, we recall the set on which the contraction is performed:
setting $A := \|u_0\|_{H^{s_c}}$, it is
\begin{equation}
\label{contraction_set_big_data}
B := \{u \in X^{s_c}([0,T]) \cap C^1_t H^{s_c}([0,T] \times X) : \|u\|_{X^{s_c}} \leq 2A, \|P_{>M} u\|_{X^{s_c}} \leq 2\delta\}.
\end{equation}
\end{proof}
We can also prove the global well-posedness of \eqref{higher_order_Hartree} for small initial data.
\begin{proof}[Proof of Proposition \ref{GWP_theorem}]
The argument is similar to the proof of \cite[Theorem 1.2]{HTT11}.
Recall that we take $V$ as in \eqref{three_body_interaction} for some two-body interaction potential $\omega \in L^1(\T)$.
We want to show that we can iterate the argument in the proof of Proposition \ref{LWP_theorem}.
It suffices to bound $\|u(t)\|_{H^1}$. In the case where $\omega > 0$, we note that
\begin{equation*}
\|u(t)\|_{H^1} \leq E(u) + \mc{M}(u) \lesssim \|u(0)\|_{H^1}^2 + \|u(0)\|_{H^1}^6.
\end{equation*}
Here we used invariance of the energy and the mass, Remark \ref{energy_finite_remark}, and the Sobolev embedding theorem. Taking $\|u(0)\|_{H^1}$ sufficiently small, we can iterate the proof of Proposition \ref{LWP_theorem}.

If there are no positivity assumptions on $\omega$, we need to use a continuity argument. Recalling \eqref{quintc_energy}, we have that
\begin{align*}
\|u\|^2_{H^1} &= E(u) + M(u) + \frac{1}{3} \int (\omega*|u|^2)^2(x)|u(x)|^2 \, \dd x \\
& \leq \|u(0)\|_{H^1}^2 + \frac{c}{3}\|\omega\|_{L^1}^2\|u(0)\|_{H^1}^6 + \frac{c}{3}\|\omega\|_{L^1}^2\|u(t)\|_{H^1}^6,
\end{align*}
where we have used conservation of energy and mass, Remark \ref{energy_finite_remark}, and the Sobolev embedding theorem. For notational simplicity, set $K:=C\|\omega\|_{L^1}^2$. We consider the function $f(x) = x - \frac{1}{3}Kx^3$. On the interval $I := [0,K^{-1/2}]$, $f$ is increasing and takes its maximum value at $\frac{2}{3}K^{-1/2}$. We note that $f(x) \geq \frac{2}{3}x$ for all $x \in I$. Let $\eps$ be as in the proof of Proposition \ref{LWP_theorem} and define
\begin{equation*}
\eps_0^2 := \min \left\{\frac{2}{3}K^{-1/2},\frac{2}{3}\eps^2\right\}.
\end{equation*}
 Consider initial data with
\begin{equation*}
\|u(0)\|_{H^1}^2 + \frac{c}{3}\|w\|_{L^1}^2\|u(0)\|_{H^1}^6 < \eps_0^2.
\end{equation*}
Using the continuity of $t \mapsto \|u(t)\|_{H^1}^2$ for $t \in [0,2\pi)$,
we have that $\|u(t)\|^2_{H^1} \in I$ for all $t \in [0,2\pi)$.
Moreover,
\begin{equation*}
\|u(t)\|_{H^1}^2 \leq \frac{3}{2}f(\|u(t)\|_{H^1}^2) \leq \frac{3}{2}\eps_0^2 \leq \eps^2.
\end{equation*}
This allows us to iterate the proof of small data well-posedness from Proposition \ref{LWP_theorem}.
\end{proof}

\section{Proof of convergence}
\label{convergence_section}
In this section, we prove Theorem \ref{Convergence_theorem}.
For notational simplicity, we prove the case where $p = 2$ and then explain how to generalise the proof for larger $p$.
\begin{proof}[Proof of Theorem \ref{Convergence_theorem} for $p=2$]
The proofs for $V_N \to \pm \delta$ are analogous, so without loss of generality we take $V_N \to \delta$.
Since the proof for negative time is the same, we also work on the interval $[0,T]$. Using 
Duhamel's formula, along with the embeddings in \eqref{embeddings}, and changing variables,
we note that \eqref{convergence_result} follows upon showing that
\begin{equation}
\label{convergence_wanted}
\lim_{N \to \infty } \left\| \left(\int_{\mathbf{X}^2} V_N(x_1,x_2) |u^N(x-x_1)|^2 |u^N(x-x_2)|^2 \, \dd x_1 \dd x_2 \right)u^N  
- |u|^4u \right\|_{N^{s_c}} = 0. 
\end{equation}
We adopt the notation
\begin{equation*}
F_{V_N}(v) := \int_{\mathbf{X}^2} V_N(x_1,x_2) |v(x-x_1)|^2 |v(x-x_2)|^2 \, \dd x_1 \dd x_2.
\end{equation*}
Then, it suffices to show that
\begin{multline}
\label{long_limit}
\underbrace{\|F_{V_N}(u^N)(u^N-u) \|_{N^{s_c}}}_{I}
+ \underbrace{\| \left(F_{V_N}(u^N) - F_{V_N}(u) \right)u \|_{N^{s_c}}}_{II}
+ \underbrace{\|\left(F_{V_N}(u) - |u|^4\right)u \|_{N^{s_c}}}_{III} \to 0
\end{multline}
as $N \to \infty$.
From the proof of local well-posedness for large initial data, we recall \eqref{contraction_set_big_data}:
\begin{equation}
\label{contraction_set_recall}
B := \{u \in X^{s_c}([0,T]) \cap C^1_t H^{s_c}([0,T] \times X) : \|u\|_{X^{s_c}} \leq 2A, \|P_{>M} u\|_{X^{s_c}} \leq 2\delta\}.
\end{equation}
Here $A = \|u_0\|_{H^{s_c}}$, $\delta = \delta(A)$ is a small constant to be determined later, and $M = M(u_0,\delta) > 1$ is such that $\|P_{>M}u\|_{X^{s_c}} < \delta$.

We will apply the decomposition $1 = P_{\leq M} + P_{>M}$
to each of $u^N(x-x_1)$, $\overline{u^N(x-x_1)}$,
and $|u^N(x-x_2)|^2$.
Note that $\overline{P_M u} = P_M\overline{u} = P_M u$.
We can therefore write
\begin{align}
\nonumber    |u^N(x-x_1)|^2|u^N(x-x_2)|^2 &= \left(P_{>M}u^N(x-x_1)\right)^2|u^N(x-x_2)|^2\\
\nonumber        &\quad+\left(P_{\leq M}u^N(x-x_1)\right)\overline{\left(P_{>M}u^N(x-x_1)\right)}
            |u^N(x-x_2)|^2\\
\nonumber        &\quad+\overline{\left(P_{\leq M}u^N(x-x_1)\right)}\left(P_{>M}u^N(x-x_1)\right)
            |u^N(x-x_2)|^2\\
\nonumber        &\quad+ \left(P_{\leq M}u^N(x-x_1)\right)^2\left(P_{> M}|u^N(x-x_2)|^2\right)\\
\label{product_rewrite}        &\quad+ \left(P_{\leq M}u^N(x-x_1)\right)^2\left(P_{\leq M}|u^N(x-x_2)|^2\right)
\end{align}

We discuss how to bound the following terms in \eqref{product_rewrite}.
\begin{align*}
    Ia &:= \int_{\mathbf{X}^2} V_N(x_1,x_2) \left(P_{>M} u^N(x-x_1) \right)^2 |u^N(x-x_2)|^2 \, \dd x_1 \dd x_2\\
    Ib &:= \int_{\mathbf{X}^2} V_N(x_1,x_2) \left(P_{\leq M}u^N(x-x_1)\right)\overline{\left(P_{>M}u^N(x-x_1)\right)}
            |u^N(x-x_2)|^2 \, \dd x_1 \dd x_2\\
    Ic &:= \int_{\mathbf{X}^2} V_N(x_1,x_2)
            \left(P_{\leq M} u^N(x-x_1) \right)^2 \left(P_{\leq M} |u^N(x-x_2)|^2\right) \, \dd x_1 \dd x_2.
\end{align*} The remaining terms in \eqref{product_rewrite} are estimated in the same way as $Ib$.

The estimate for $Ia$ follows from \eqref{multilinear_estimate}. Indeed, we have
\[
    \|(Ia)(u^N-u)\|_{N^{s_c}} 
    \lesssim \|V_N\|_{L^1}\|P_{>M}u^N\|_{X^{s_c}}^2\|u^N\|_{X^{s_c}}^2\|u^N-u\|_{X^{s_c}}.
\]
Recalling that $\|V_N\|_{L^1} = 1$, we can pick $\delta$ sufficiently small that
\begin{equation}
\label{Ia_bound}
    \|(Ia)(u^N-u)\|_{N^{s_c}} \leq \frac{1}{32} \|u^N-u\|_{X^{s_c}}.
\end{equation}

In a similar way, we have \[
    \|(Ib)(u^N-u)\|_{N^{s_c}} \lesssim \|P_{> M}u^N\|_{X^{s_c}}\|P_{\leq M}u^N\|_{X^{s_c}}\|u^N\|^2_{X^{s_c}}\|u^N-u\|_{X^{s_c}}.
\] Using the boundedness of $P_{\leq M}$, we obtain
\begin{equation}
\label{I_b_bound}
    \|(Ib)(u^N-u)\|_{N^{s_c}} \leq \frac{1}{32} \|u^N-u\|_{X^{s_c}},
\end{equation} 
provided that $\delta$ is small enough.

Finally, by H\"older's inequality, \eqref{Ns_embedding}, and \eqref{embeddings} we have
\begin{align*}
%\label{IC_bound}
\|(Ic)(&u^N-u)\|_{N^{s_c}} \\
&\lesssim T \int_{\mathbf{X}^2} V(x_1,x_2)\|P_{\leq M} |u^N(\cdot-x_1)|^2 P_{\leq M} |u^N(\cdot-x_2)|^2\|_{L^\infty_{t,x}} \dd x_1 \, \dd x_2\nonumber \|(u^N-u)\|_{L^\infty_t H^{s_c}_x}\\
    &\lesssim T \|P_{\leq M}u^N\|_{L^\infty_{t,x}}^4\|u^N-u\|_{X^{s_c}}.
\end{align*}
Here we recall that $V_N > 0$. Using Strichartz's and Bernstein's estimates, we have the bounds
\begin{equation}\label{eqn:PleqMinfty}
    \|P_{\leq M}u^N\|_{L^\infty_{t,x}}^4
        \lesssim M^{4(\frac{3}{2}-s_c)}\|P_{\leq M}u^N\|_{X^{s_c}}^4
        \lesssim M^2(2A)^4.
\end{equation}
Picking $T$ sufficiently small depending on $M$, $A$, and $\delta$, we find
\begin{equation}
\label{Ic_bound}
\|(Ic)(u^N-u)\|_{N^{s_c}} \leq \frac{1}{32}\|u^N-u\|_{X^{s_c}}.
\end{equation}
Combining \eqref{Ia_bound}, \eqref{I_b_bound}, and \eqref{Ic_bound}, we have
\begin{equation}
\label{I_bound}
I \leq \frac{1}{8}\|u^N-u\|_{X^{s_c}}.
\end{equation}
For term $II$ in \eqref{long_limit}, we use the identity
\begin{multline}
\label{convergence_identity}
|u^N(x-x_1)|^2 |u^N(x-x_2)|^2 - |u(x-x_1)|^2 |u(x-x_2)|^2 \\
= |u^N(x-x_1)|^2 \left(|u^N(x-x_2)|^2 - |u(x-x_2)|^2\right)  \\
+ \left((|u^N(x-x_1)|^2 - |u(x-x_1)|^2)\right)|u(x-x_2)|^2.
\end{multline}
Combining \eqref{convergence_identity} with the identity $|u|^2 - |v|^2 = u\overline{(u-v)} + (u-v)\overline{v}$, we see that we can estimate term $II$ by arguing as in \eqref{I_bound}. So
\begin{equation}
\label{II_bound}
II \leq \frac{1}{8}\|u^N-u\|_{X^{s_c}}.
\end{equation}
For term $III$ in \eqref{long_limit}, we have
\begin{multline}
\label{delta_convolution_decomposition}
\int_{\mathbf{X}^2} V_N(x_1,x_2) |u(x-x_1)|^2 |u(x-x_2)|^2 \, \dd x_1 \dd x_2 - |u(x)|^4 \\
= \int_{\mathbf{X}^2} V_N(x_1,x_2) \bigg[\bigg(|u(x-x_1)|^2 - |u(x)|^2\bigg)|u(x)|^2 \\
+ |u(x-x_1)|^2\bigg(|u(x-x_2)|^2 - |u(x)|^2\bigg)  \bigg] \dd x_1 \dd x_2,
\end{multline}
where we used $\int_{\mathbf{X}^2} V_N(x_1,x_2) = 1$. Using \eqref{delta_convolution_decomposition}, we split $III$ from \eqref{long_limit} into two terms. Both terms are bounded in the same way, so we only detail the first. Namely, we bound
\begin{equation}
\label{III_bound_1}
\left\|\int V_N(x_1,x_2) \left(|u(x-x_1)|^2 - |u(x)|^2\right) \, \dd x_1 \dd x_2 |u|^2 u \right\|_{N^{s_c}}.
\end{equation}
Since $V_N \to \delta$, we can assume without loss of generality that $\mathrm{supp}(V_N) \subset [-\frac{1}{N},\frac{1}{N}]^3$. Applying Proposition \ref{multilinear_lemma}, we have
\begin{equation}
\label{III_bound}
\eqref{III_bound_1} \lesssim \sup_{[-\frac{1}{N},\frac{1}{N}]^3} \|u(\cdot) - u(\cdot - x_1)\|_{X^{s_c}}\|u\|^4_{X^{s_c}} \to 0
\end{equation}
by the local well-posedness of the local NLS.
Combining \eqref{I_bound}, \eqref{II_bound}, and \eqref{III_bound}, we have proved \eqref{convergence_wanted}. To get the result on $[0,T]$, use an inductive argument, similar to the proofs of \cite[Lemmas 5.4 and 5.7]{RS23}.
\end{proof}
\begin{remark}
In the case that \eqref{p_NLS} is globally well-posed, we can use an induction argument to extend the convergence argument to arbitrary time for each fixed $u_0$. See the proofs of \cite[Lemmas 5.4 and 5.7]{RS23} for a similar argument.
\end{remark}
\begin{remark}
We now explain how to generalise the proof of Theorem \ref{Convergence_theorem} to higher-order nonlinearities. We now write
\begin{equation*}
F_{V_N}(u) := \int_{\mathbf{X}^p} V_{N}(x-x_1,\ldots,x-x_p) |u(x_1)|^2\ldots|u(x_p)|^2 \, \dd x_1 \ldots \dd x_p.
\end{equation*}
We consider $B$ as in \eqref{contraction_set_recall}, although possibly for a different choice of $\delta$ and $M$.
As in \eqref{long_limit}, we split into the three terms $I$, $II$, and $III$.
For $I$, we iterate the decomposition $1 = P_{>M} + P_{\leq M}$, applying it to each $|u^N(x-x_j)|^2$.
Every term which features a $P_{>M}$ is estimated as $\|(Ia)(u^N-u)\|_{N^{s_c}}$ was estimated.
The one term featuring only $P_{\leq M}$ is estimated as $\|(Ic)(u^N-u)\|_{N^{s_c}}$ was estimated.
Note that, by \eqref{eqn:PleqMinfty} and \eqref{critical_sobolev},
we always get a term of $M^2$, regardless of the value of $p\geq 2$.
By picking $\delta$ and $T$ sufficiently small, one recovers \eqref{I_bound}. In this case, instead of picking $\frac{1}{32}$, we pick a small enough constant so that \eqref{I_bound} still holds with constant $\frac{1}{8}$. 

The bounds on $II$ and $III$ follow by applying decompositions analogous to
\eqref{convergence_identity} and \eqref{delta_convolution_decomposition},
and arguing as in the proofs of \eqref{II_bound} and \eqref{III_bound}, respectively.
\end{remark}

\section{Mixed Hartree nonlinearities}
\label{mixed_nl_section}
In this section, we consider the case of mixed Hartree-type nonlinearities. Recall that for $p \in \N$, we write
\begin{equation*}
\label{5.mathcal_N}
\mathcal{N}_p(u) := \left(\int_{\mathbf{X}^p} V(x-x_1,\ldots,x - x_p) |u(x_1)|^2 \ldots |u(x_p)|^2 \, \dd x_1 \ldots \dd x_p, \right)u(x).
\end{equation*}
and consider the mixed-Hartree equation given by
\begin{equation}
\label{5.mixed_nonlinearity}
\ii \partial_t u + \Delta u = \lambda_1 \mathcal{N}_{p_1}(u) + \lambda_2 \mathcal{N}_{p_2}(u).
\end{equation}
\subsection{Energy and multilinear estimates}
In this section, we state and prove some estimates we use to prove the well-posedness of \eqref{5.mixed_nonlinearity}.

To consider the global well-posedness of \eqref{5.mixed_nonlinearity} in the case of a cubic-quintic nonlinearity, we need estimates on the 
kinetic energy of the solution. In other words, we need to bound the size of $\|\nabla u\|^2_{L^2_x L^\infty_I}$.
Throughout the remainder of this section, we adopt the convention that $V_{1}(x-y) = \omega'(x-y)$ and $V_{2}$ is as in \eqref{three_body_interaction},
for some even two-body interaction potentials $\omega,\omega' \in L^1(\mathbf{X})$.
Then the Hamiltonian is given by
\begin{multline}
\label{Hamiltonian_mixed_nonlinearity}
E(u) = \int_{\mathbf{X}} |\nabla u(x)|^2\dd x + \frac{\lambda_1}{2} \int_{\mathbf{X}} V_1(x-y) |u(x)|^2|u(y)|^2 \, \dd x \, \dd y \\
+ \frac{\lambda_2}{3} \int_{\mathbf{X}^2} V_2(x-y,x-z) |u(x)|^2||u(y)|^2|u(z)|^2 \, \dd x \, \dd y \, \dd z .
\end{multline}

We note that $E(u)$ is finite for $u\in H^1(\mathbf{X})$ by Lemma \ref{energy_finite_lemma} and Remark \ref{energy_finite_remark}. We can now prove the following bound on the kinetic energy.
\begin{lemma}
\label{energy_bound_lemma}
Suppose that $w \in L^1$ is even, $\omega'=\omega$, $\lambda_1 \in \R$, and $\lambda_2 > 0$. Then 
\begin{equation}
\label{mixed_nonlinearity_energy_bound}
\|\nabla u\|^2_{L^\infty_t L^2_x} \lesssim E(u) + \mathcal{M}(u).
\end{equation}
\end{lemma}
\begin{remark}
Let us remark that if $\lambda_1, \lambda_2 \geq 0$, we trivially have \eqref{mixed_nonlinearity_energy_bound} by the definition of $E$.
\end{remark}
\begin{proof}[Proof of Lemma \ref{energy_bound_lemma}]
We note that we can write
\begin{multline*}
\frac{\lambda_1}{2}\int_{\mathbf{X}^2} \omega(x-y)  |u(x)|^2 |u(y)|^2 \, \dd x \, \dd y + \frac{\lambda_2}{3} \int_{\mathbf{X}^3} V_2(x-y,x-z)|u(x)|^2 |u(y)|^2|u(z)|^2 \,  \dd x \, \dd y \, \dd z \\
= \int_{\mathbf{X}}  \left[ \frac{\lambda_2}{3}(\omega*|u|^2)^2|u(x)|^2 + \frac{\lambda_1}{2} (\omega*|u|^2)|u(x)|^2\right] \dd x .
\end{multline*}
So there exists a $C > 0$ such that \eqref{mixed_nonlinearity_energy_bound} holds.
\end{proof}

We will use the following multilinear estimate to prove local well-posedness for the mixed Hartree equation.
\begin{proposition}
\label{mixed_multlinear_proposition}
Suppose $p_2>p_1$, with $p_i \in \N$. Let $s_c = s_c(p_2)$ be as in \eqref{critical_sobolev} with $p=p_2$, and suppose that $s>s_c$. For $u_j \in X^{s}(I)$, $j \in \{1,\ldots,2p_1+1\}$, the following estimate holds.
\begin{multline}
\label{multlinear_mixed_estimate}
\bigg\| \int_{\mathbf{X}^{p_1}} \bigg[ V_{p_1}(x-x_1,\ldots,x-x_{p_1}) \tilde{u}_{1}(x_1,t) \tilde{u}_{2}(x_1,t) \ldots \tilde{u}_{2p_1-1}(x_p,t) \tilde{u}_{2p_1}(x_{p_1},t) \bigg] \\
\times \dd x_1 \ldots \dd x_{p_1} \, \tilde{u}_{2p_1+1}(x,t) \bigg\|_{N^{s_c(p_2)}(I)} \\ \lesssim |I|^{c}\|V_{p_1}\|_{L^1} 
\prod^{2p_1+1}_{k=1} \|u_k\|_{X^{s_c(p_2)}(I)}.
\end{multline}
Here $\tilde{u}_j$ denotes either $u_j$ or $\overline{u}_j$ and $0<c<1$.
\end{proposition}

\begin{proof}
    The proof is essentially the same as that of Lemma \ref{multilinear_lemma};
    we simply outline how to extract the power of $|I|^c$.

    Suppose that $M_1 \geq M_2$ are dyadic numbers, and let $\mathcal{C}_2$ be a cube of size $M_2$.
    Then, applying H\"older's inequality and the Strichartz estimate (\ref{Strichartz_estimate}), we have
    \begin{align*}
        \|P_{\mathcal{C}_2}P_{M_1}u_1\|_{L^4_{t,x}} &\lesssim
            |I|^{\frac{q-4}{4q}}\|P_{\mathcal{C}_2}P_{M_1}u_1\|_{L^q_{t,x}}\\
            &\lesssim |I|^{\frac{q-4}{4q}} M_2^{\frac{3}{2}-\frac{5}{q}}\|P_{\mathcal{C}_2}P_{M_1}u_1\|_{Y^0}
    \end{align*} for any $q > \frac{10}{3}$.
    Summing over cubes $\mathcal{C}_2$ we obtain, by orthogonality, the inequality
    \begin{equation}\label{PM1u1L4}
         \|P_{M_1}u_1\|_{L^4_{t,x}} \lesssim |I|^{\frac{q-4}{4q}} M_2^{\frac{3}{2}-\frac{5}{q}}\|P_{M_1}u_1\|_{Y^0},
    \end{equation}
    and we choose $q$ to ensure that the factor of $M_2^{\frac{3}{2}-\frac{5}{q}}$
    balances the calculations in (\ref{M1M2balance_case1}) and (\ref{M1M2balance_case2}).
    Note that $c = \frac{q-4}{4q}<1$ in any case.
\end{proof}

When $p_1 = 1$ and $p_2 = 2$ --- so that $s_c(p_2) = 1$ --- we have the following
multilinear estimates involving the interpolation norm $\|\cdot\|_{Z'(I)}$.
\begin{proposition}\label{multilinearZ_estimates}
If $u_j \in X^{1}(I)$, $j \in \{1,\ldots,3\}$, then the following estimate holds:
\begin{multline}\label{multilinear_estimate_p5}
\bigg\| \int_{\mathbf{X}^{2}} \bigg[V(x-y,x-z) \tilde{u}_{1}(y,t) \tilde{u}_{2}(y,t) \tilde{u}_{3}(z,t) \tilde{u}_4(z,t)\bigg]
    \dd y \, \dd z \, \tilde{u}_{5}(x,t) \bigg\|_{N^{1}(I)}\\
    \lesssim \|V\|_{L^1} \sum_{j=1}^5\|u_j\|_{X^{1}(I)}\prod_{k\neq j}\|u_k\|_{Z'(I)}.
\end{multline}
\begin{equation}\label{multilinear_estimate_p3_I}
\bigg\| \int_{\mathbf{X}} \bigg[V(x-y) \tilde{u}_{1}(y,t) \tilde{u}_{2}(y,t) \bigg] \dd y \, \tilde{u}_{3}(x,t) \bigg\|_{N^{1}(I)} \lesssim |I|^{\frac{1}{5}}\|V\|_{L^1} 
\prod_{j=1}^3\|u_j\|_{X^{1}(I)}.
\end{equation}
Here $\tilde{u}_j$ denotes either $u_j$ or $\overline{u}_j$.
\end{proposition}
\begin{proof}
    The proof of (\ref{multilinear_estimate_p5}) can be found, for example, in \cite[Lemma 23]{NS20},
    while (\ref{multilinear_estimate_p3_I}) is a special case of (\ref{multlinear_mixed_estimate})
    with $p_1=1$, $p_2=2$, and $q=20$.
\end{proof}
\subsection{Well-posedness of mixed Hartree-equation for general $p$}
We now prove Propositions \ref{mixed_nonlinearity_LWP} and \ref{mixed_nonlinearity_GWP}.
\begin{proof}[Proof of Proposition \ref{mixed_nonlinearity_LWP}]
We argue similarly to the proof of Proposition \ref{LWP_theorem}, using the multilinear estimates from Proposition \ref{mixed_multlinear_proposition} as well as Proposition \ref{multilinear_lemma}. Our contraction is done on the set
\begin{equation*}
B_1 := \{u \in X^{s_c}([0,T)) \cap C_t([0,T);H^{s_c}(\T^3)) : \|u\|_{X^{s_c}([0,T))} \leq 2\eps\},
\end{equation*}
for $\eps$ sufficiently small. We omit the details.
\end{proof}

\begin{proof}[Proof of Proposition \ref{mixed_nonlinearity_GWP}]
In the case that $\lambda_i > 0$ and $\omega,\omega' \geq 0$, the energy is positive, so we argue as in the defocusing case of the proof of Proposition \ref{GWP_theorem}. If $\lambda_2 > 0$, $\omega=\omega'$, and there is no positivity assumption on $\omega$, we can use \eqref{mixed_nonlinearity_energy_bound} and argue similarly to the proof of the defocusing case of the proof of Proposition \ref{GWP_theorem}. We omit the details. In the case that $\lambda_2 < 0$ or $\omega \neq \omega'$, we explain how to adapt the proof of the focusing case of the proof of Proposition \ref{GWP_theorem}. We now have
\begin{multline*}
\|u\|^2_{H^1} = E(u) + M(u) \\
+ \frac{|\lambda_1|}{2} \int_{\mathbf{X}} (\omega'*|u|^2)(x)|u(x)|^2 \, \dd x + \frac{|\lambda_2|}{3} \int_{\mathbf{X}} (\omega*|u|^2)^2(x)|u(x)|^2 \, \dd x \\
 \leq \|u(0)\|_{H^1}^2 + \frac{c|\lambda_1|}{2}\|\omega'\|_{L^1}\|u(0)\|_{H^1}^4 +  \frac{c|\lambda_2|}{3}\|\omega\|_{L^1}^2\|u(0)\|_{H^1}^6 \\
 + \frac{c|\lambda_1|}{2}\|\omega'\|_{L^1}\|u(t)\|_{H^1}^4 + \frac{c|\lambda_2|}{3}\|\omega\|_{L^1}^2\|u(t)\|_{H^1}^6.
\end{multline*}
Here we have used Lemma \ref{energy_finite_lemma} and Remark \ref{energy_finite_remark}. We set $K_1 := c|\lambda_1|\|\omega'\|_{L^1}$, $K_2 := c|\lambda_2|\|\omega\|_{L^1}^2$. We now consider the function $f(x) = x - \frac{1}{2}K_1x^2 - \frac{1}{3}K_2x^3$.
Let \[
    x^*:=\frac{-K_1 + \sqrt{K_1^2 + 4K_2}}{2K_2} > 0.
\]
Then $f$ increases on the the interval $I := [0,x^*]$ up to its maximum value $f(x^*)$. We also note that $f(x) \geq \frac{1}{2}x$ on $[0,x^*]$. Let $\eps$ be as in the proof of Proposition \ref{mixed_nonlinearity_LWP}. We set
\begin{equation*}
\eps_0 := \min \left\{f(x^*), \frac{1}{2}\eps^2  \right\}.
\end{equation*}
Consider initial data with
\begin{equation*}
\|u(0)\|_{H^1}^2 + \frac{c|\lambda_1|}{2}\|\omega'\|_{L^1}\|u(0)\|_{H^1}^4 +  \frac{c|\lambda_2|}{3}\|\omega\|_{L^1}^2\|u(0)\|_{H^1}^6 < \eps_0^2.
\end{equation*}
Using the continuity of $t \mapsto \|u(t)\|_{H^1}^2$ for $t \in [0,2\pi)$, we have $\|u(t)\|^2_{H^1} \in I$ for all $t \in [0,2\pi)$. Moreover
\begin{equation*}
\|u(t)\|_{H^1}^2 \leq 2f(\|u(t)\|_{H^1}^2) \leq 2\eps_0^2 \leq \eps^2.
\end{equation*}
This allows us to iterate the proof of small data well-posedness from Proposition \ref{mixed_nonlinearity_LWP}.
\end{proof}

\subsection{Global well-posedness for all data for perturbations of the quintic NLS}
In this section, we prove Theorem \ref{GWP_mixed_nonlinearity_close_quintic} and consider the cases where \eqref{H1_estimate} holds.

\begin{proof}[Proof of Theorem \ref{GWP_mixed_nonlinearity_close_quintic}]
The proof is similar to the proof of \cite[Theorem 1.2]{LYYZ23}. We explain the differences in our case.
Fix $u_0 \in H^1(\mathbf{X})$ and consider a subinterval $ J = (a,b) \subset (-T,T) =: I$.
Let us assume for now that $|J|\leq 1$ and is sufficiently small --- to be determined later.
As in the proof of \cite[Theorem 1.2]{LYYZ23}, we use the well-posedness of the energy-critical quintic NLS from \cite{IP12b}.
Specifically, on any interval $I' = (-T',T')$ with $|I'|\leq 1$, the problem
\begin{equation}
\label{quintic_NLS}
\left\{
\begin{alignedat}{2}
\ii \partial_t v + \Delta v &= |v|^4v, \\
v_0 &\in H^1(\mathbf{X}),
\end{alignedat}\right.
\end{equation}
has a unique solution $v\in X^1(I')$ satisfying the following energy estimates:
\begin{equation}
\label{X_Z_norm_bound}
\|v\|_{X^1(I')} + \|v\|_{Z'(I')} \leq C(M(v_0),E(v_0)) < \infty.
\end{equation}

Consider the difference equation given by
\begin{equation}
\label{difference_equation}
\ii \partial_t w + \Delta w = \mathcal{N}_1(w + v) + \mathcal{N}_{2,N}(w+v) - |v|^4v.
\end{equation}
We now follow the construction in \cite{LYYZ23}, partitioning $I$ into $m+1$ subintervals $I_j = [t_{j},t_{j+1}]\subset I$
of length at most 1 and such that \[
    \|v\|_{Z'(I_j)} \leq \eta,
\] where $\eta>0$ will be specified later.
We restrict our attention to those $I_j$ which intersect $J=(a,b)$,
and assume that the solution $w$ to \eqref{difference_equation} (started at $t_{j-1}$) satisfies the bound
\begin{equation}\label{maxinduction}
     \max\{\|w\|_{L_t^\infty(I_{j-1}) H^1_x}, \|w\|_{X^1(I_{j-1})}\} \leq (2C)^{j-1}|J|^{\frac{1}{20}},
\end{equation} for some absolute constant $C>0$ independent of $j$ and $\eta$. In particular,
\begin{equation}
\label{w_tj_induction}
\|w(t_j)\|_{H_x^1} \leq \|w\|_{L_t^\infty(I_{j-1}) H^1_x} \leq (2C)^{j-1}|J|^{\frac{1}{20}}.
\end{equation}
We now construct a solution on $I_j$ with initial condition $w(t_j)$. Define
\begin{equation*}
\label{S_j}
\mc{S}_j := \{w \in X^1(I_j) : \max\{\|w\|_{L_t^\infty(I_j) H^1_x}, \|w\|_{X^1(I_j)}\} \leq (2C)^{j}|J|^{\frac{1}{20}}\}.
\end{equation*}
We show that the operator \[
    L_jw := \e^{\ii(t-t_j)\Delta}w(t_j) - \ii\int_{t_j}^{t}\e^{\ii(t-s)\Delta}\left(
        \mathcal{N}_1(v+w) + \mathcal{N}_{2,N}(w+v) -|v|^4v\right)(s)\,\dd s
\] is a contraction on $\mathcal{S}_j$.

We now show that $L_j$ maps $\mathcal{S}_j$ to itself,
using the assumption that $V_{2,N} \to \delta_2$. Indeed, arguing similarly to \eqref{III_bound},
we can choose $N = N(\|u(a)\|_{H^1})$ large enough and the support of $V_{2,N}$ small enough that
\begin{equation}\label{delta_Veps_bound}
    \int_{\mathbf{X}^2}V_{2,N}(x_1,x_2)\left[|v(x-x_1)|^2|v(x-x_2)|^2v(x) - |v|^4v(x)\right]\dd x_1\,\dd x_2 < \frac{1}{20}(2C)^j|J|^\frac{1}{20}.    
\end{equation}
This controls the highest-order terms in $v$ in the nonlinearity. For $w \in \mc{S}_j$, by \eqref{maxinduction}, we have
\begin{align}
    \max \{ &\|L_j w\|_{L_t^\infty(I_j) H_x^1}, \|L_j w\|_{X^1(I_j)} \}\nonumber\\
        &\leq \|w(t_j)\|_{H_x^1} + \frac{1}{20}(2C)^j|J|^\frac{1}{20} + C_0\|w\|_{X^1(I_j)}^5\nonumber\\
            &\qquad + C_0 \sum_{k=1}^{4}\left(\|w\|_{X^1(I_j)}^{5-k}\|v\|_{Z'(I_j)}^k
                +\|v\|_{X^1(I_j)}\|w\|_{X^1(I_j)}^{5-k}\|v\|_{Z'(I_j)}^{k-1}\right)\nonumber\\
            &\qquad + C_0|J|^{1/5}\left(\|w\|_{X^1(I_j)}+\|v\|_{X^1(I_j)}\right)^3\nonumber\\
        &\leq  (2C)^{j-1}|J|^{\frac{1}{20}} + \frac{1}{20}(2C)^j|J|^\frac{1}{20}
                 +  C_0 \left[(2C)^j|J|^\frac{1}{20}\right]^5 \label{Lj_easyterms}\\
            &\qquad + C_0 \sum_{k=1}^{4}\left[(2C)^j|J|^\frac{1}{20}\right]^{5-k}\eta^k
                    + C_0 \|v\|_{X^1(I_j)}\sum_{k=1}^{4}\left[(2C)^j|J|^\frac{1}{20}\right]^{5-k}\eta^{k-1} \label{horrendous_mess}\\
            &\qquad + C_0 |J|^{1/5}\left(\|w\|_{X^1(I_j)} +\|v\|_{X^1(I_j)}\right)^3 \label{cubic_terms}.
\end{align}
The first two terms in \eqref{Lj_easyterms} come from \eqref{maxinduction}, \eqref{w_tj_induction},
and the estimate \eqref{delta_Veps_bound}.
The remaining terms arise by an application of the multilinear estimates in Proposition \ref{multilinearZ_estimates},
the embedding $Z'\hookrightarrow X^1$ and \eqref{maxinduction}.

We now make $\eta < \frac{1}{20C_0}$, which ensures that all terms in \eqref{horrendous_mess}
are small enough.
Provided that $(2C)^j|J|^{\frac{1}{20}} \leq 1$ (shrinking $J = J(m,\|v\|_{X^1(I_j)})$ if necessary),
we can make all the terms in \eqref{Lj_easyterms} and \eqref{horrendous_mess}
bounded by $\frac{1}{20}(2C)^j|J|^{\frac{1}{20}}$.
(Recall that $|J|\leq 1$ and $j \leq m =  m(\|v\|_{Z'(I)})$ and that $\|v\|_{Z'(I)}$ is controlled by
the initial data $\|u(a)\|_{H^1_x}$.
For the terms in \eqref{cubic_terms},
we again shrink $J$ if necessary,
recalling that $\|v\|_{X^1(I_j)}$ is controlled by $\|u(a)\|_{H^1_x(I)}$, independently of $J$.
So \[
    \max \{ \|L_j w\|_{L_t^\infty(I_j) H_x^1}, \|L_j w\|_{X^1(I_j)} \} \leq (2C)^j|J|^{\frac{1}{20}}.
\]
This proves $L_j$ maps $\mathcal{S}_j$ to itself.
A similar argument shows that $L_j$ is a contraction with respect to the $X^1(I_j)$ norm, and the smallness assumption that $|J|\leq 1$ can be removed as in the proof of \cite[Theorem 1.2]{LYYZ23}.
\end{proof}

For the remainder of this subsection, take $V_{2,N}$ of the form of \eqref{three_body_interaction},
so
\begin{equation*}
V_{2,N}(x-y,x-z) = \frac{1}{3}(\omega_N(x-y)\omega_N(x-z) + \omega_N(x-y)\omega_N(y-z) + \omega_N(x-z)\omega_N(y-z)).
\end{equation*}
Here we define $\omega_N$ in the following way. Let $\Psi\colon \R^3 \to \R$ be a positive continuous even function with compact support and $\int_{\R^3} \Psi(x) \, \dd x = 1$. For $N \in \N$, we define
\begin{equation*}
\omega_N(x) := N^3 \Psi \left(N[x]\right),
\end{equation*}
where $[x]$ denotes the unique element in $x + \Z^3 \cap \T^3$.
Then, the function $\omega_N$ lies in $L^\infty$ and is even.
Moreover, one has that $\omega_N \to \delta$ with respect to continuous functions. We note the following properties of $\omega_N$:
\begin{equation}
\label{weps_properties_bounded}
|\hat{\omega}_N(k) - 1| < C,
\end{equation}
for any $k \in \Z$ and for any $N \in \N$; and moreover
\begin{equation}
\label{w_eps_properties_convergence}
|\hat{\omega}_N(k) - 1| \to 0
\end{equation}
as $N \to \infty$ for all $k \in \Z$. We make the following remark about the application of Theorem \ref{GWP_mixed_nonlinearity_close_quintic}.
\begin{remark}
\label{H1_estimate_remark}
It is not immediately clear when \eqref{H1_estimate} is satisfied. We record the following special cases.
\begin{enumerate}
\item \label{epscial_case_1} Suppose that, in the statement of Theorem \ref{GWP_mixed_nonlinearity_close_quintic}, one takes $\omega' = \omega_N$. Then it follows from Lemma \ref{energy_bound_lemma} that \eqref{H1_estimate} is true.
\item When $\omega' \geq 0$, \eqref{H1_estimate} is trivially satisfied.
\item \label{special_case_2} Suppose that we restrict our initial data in the statement of Theorem \ref{GWP_mixed_nonlinearity_close_quintic} to a ball of finite radius around zero in $H^1$.
Then, there exists an $N' > 0$ such that for any $N > N'$, we have \eqref{H1_estimate}.
This is proved in Lemma \ref{ball_H^1_energy_bound_lemma} below.
\end{enumerate}
\end{remark}
\begin{lemma}
\label{ball_H^1_energy_bound_lemma}
Consider the ball $B_C(0) := \{u \in H^1(\mathbf{X}) : \|u\|_{H^1} < C \}$. Then there is some $N' > 0$ such that for all $N > N'$, \eqref{H1_estimate} is true for $E(u) \equiv E_N(u)$.
\end{lemma}
\begin{proof}
Let $u \in B_C(0)$. By considering a density argument, we assume without loss of generality that $u$ and $\omega_N$ are smooth. Consider
\begin{equation*}
G_N(x) := \left|\omega_N * |u|^2 - |u|^2\right|(x) = \big|\sum_{{\substack{ k \in \Z \\ k_1 +k_2 =k} }} (\hat{\omega}_N(k) - 1) \hat{u}(k_1) \hat{\bar{u}}(k_2)\e^{\ii k x}\big|.
\end{equation*}
By \eqref{weps_properties_bounded}, it follows that $|G_N(x)| \leq C \|u\|_{L^2}^2 \lesssim \|u\|_{H^1}^2$.
Applying the dominated convergence theorem and \eqref{w_eps_properties_convergence}, and taking $N$ sufficiently large, for some $\eta > 0$ sufficiently small, we have
\begin{equation}
\label{Geps_convergence}
|G_N(x)| \lesssim \eta \|u\|_{H^1}^2 \leq \frac{1}{4}.
\end{equation}
It follows from \eqref{Geps_convergence}, H\"older's inequality, and \eqref{energy_finite_remark}
(as well as $\|\omega_N\|_{L^1} = \int \omega_N(x-y) \, \dd y = 1$) that for $N$ sufficiently large, we have
\begin{equation*}
\left|\frac{1}{3}\int (\omega_N * |u|^2)^2 |u(x)|^2 - |u(x)|^6 \, \dd x  \right| \leq \frac{1}{2}\|u\|_{L^4}^4.
\end{equation*}
Arguing as in the proof of Lemma \ref{energy_bound_lemma}, we obtain \eqref{H1_estimate}.
\end{proof}

\subsection{Convergence of the mixed Hartree equation to the mixed NLS}
We now prove Corollary \ref{mixed_convergence_theorem}.

\begin{proof}[Proof of Corollary \ref{mixed_convergence_theorem}]
The proof is similar to the proof of Theorem \ref{Convergence_theorem}. Now, in \eqref{long_limit} we will have three extra terms, corresponding to the lower order nonlinearity. These three extra terms are treated analogously to the higher order terms, except that we use Proposition \ref{mixed_multlinear_proposition} instead of Proposition \ref{multilinear_lemma}. We make the argument on a small interval $I$, so that $|I|^c < 1$. We then make an inductive argument similar to the proofs of \cite[Lemmas 5.4 and 5.7]{RS23} to build up all the way to $T'$.
\end{proof}

\appendix 
\section{Specific interaction potentials}
\label{appendix_hamiltonian}
In this appendix, we comment on particular forms of the interaction potential $V$, which can be useful when exploiting the Hamiltonian structure of \eqref{p_Hartree}.
\subsection{Quintic nonlinearity}
We first discuss the case of the quintic nonlinearity. In this case, the quintic Hartree equation arises as the mean-field limit of the many-body Schr\"odinger equation with Hamiltonian given by
\begin{equation*}
H_N = \sum_{j=1}^N -\Delta_j + \frac{\mu}{N^2}\sum_{i<j<k}^N V(x_i-x_j,x_i-x_k),
\end{equation*}
see for example \cite{CH19,NS20,Xie15}. Here $\mu = \pm 1$. Let $\omega \colon \mathbf{X} \to \R$ be an even, real-valued interaction potential. Then two natural choices of $V$ are
\begin{equation}
\label{pairwise_sum}
V_1:= \frac{1}{3}\left(\omega(x_i-x_j)\omega(x_i-x_k) + \omega(x_i - x_j)\omega(x_j - x_k) + \omega(x_i-x_k) \omega(x_j - x_k)\right),
\end{equation}
or
\begin{equation}
    \label{triple_product}
V_2:= \omega(x_i-x_j)\omega(x_j-x_k)\omega(x_k-x_i).
\end{equation}
Interaction potentials of the form $V_1$ were considered in for example \cite{Lee20,NS20}, and of the form $V_2$ in \cite{RS23}. We direct the reader to \cite[Section 1 and Appendix A]{RS23} for the Hamiltonian structure of \eqref{p_Hartree} corresponding to the choices $V_1$ and $V_2$. In the case of $V_2$, the energy is given by
\begin{equation}
\label{3_cluster_energy}
E(u) = \int |\nabla u (x)|^2 \, \dd x + \frac{1}{3}\int \omega(x-y)\omega(y-z)\omega(z-x) |u(x)|^2 |u(y)|^2 |u(z)|^2 \, \dd x \, \dd y \, \dd z.
\end{equation}
The finiteness of the energy in this case for $u \in H^1(\mathbf{X})$ and $\omega$ even and in $L^{\frac{3}{2}}$ can be proved similarly to \cite[Lemma A.9]{RS23}. The only difference to the proof of \cite[Lemma A.9]{RS23} is we embed into $H^1$ when using the Sobolev embedding theorem, since we are working in three dimensions instead of one.

We now consider the problem of convergence for these particular choices of interaction potential. Fix
\begin{equation}
    \label{w_N}
\omega^*_n(x) := n^3\chi_{\left[-\frac{1}{2n},\frac{1}{2n}\right]^3}(x),
\end{equation}
where $\chi_A$ denotes the indicator function on a set $A$. We note that $\omega^*_n$ converges to the delta distribution with respect to continuous functions as $n \to \infty$. For $\eqref{pairwise_sum}$, we set $\omega_n := \omega^*_n$, and by direct computation, we have
\begin{equation*}
\|V_{1,n}\|_{L^1} = 1, \quad V_{1,n} \to \delta,
\end{equation*}
Here we write $V_{i,n}$ for $V_i$ defined using $\omega_n$.
Then we are able to apply Theorem \ref{Convergence_theorem} for the sequence $V_{1,n}$.

In the case of \eqref{triple_product}, we note that we cannot take $\omega_n = c \omega^*_n$ for any constant $c>0$. Indeed, if we do this, a simple computation shows
\begin{equation*}
\|V_{2,n}\|_{L^1} \gtrsim n^3 \to \infty.
\end{equation*}
Instead, we take $\omega_n := \left(\omega^*_n\right)^{\frac{2}{3}}$. In this case, a computation shows that
\begin{equation*}
\|V_{2,n}\|_{L^1} = \frac{1}{512}.
\end{equation*}
Rescaling $V_{2,n}$ appropriately, one can show $V_{2,n} \to \delta$ and $\|V_{2,n}\|_{L^1} = 1$, so we can apply Theorem \ref{Convergence_theorem}.

\subsection{Higher order nonlinearities}
In the case of a higher-order nonlinearity, one can again split into two different cases. 
\subsubsection*{Generalising $V_1$}
The case of $V_1$ is generalised by taking the sums of products of $p$--two body interactions as follows:
\begin{equation}
\label{generalised_three_body_interaction}
V_{1,p}(x-x_1,\ldots,x-x_p) := \frac{1}{p+1}\sum_{i=0}^p \prod_{\substack{j=0 \\ j \neq i}}^p \omega(x_i-x_j),
\end{equation}
where again we write $x \equiv x_0$. We interpret this as every particle interacting with each other particle by means of a two-particle interaction potential. Recall we take $\mu \in \{\pm 1\}$. The mean-field NLS one expects to be associated to
\begin{equation*}
H_N = \sum_{j=1}^N -\Delta_j + \frac{\mu}{N^p}\sum^N_{1 \leq i_0 < i_1 < \ldots < i_{p} \leq N} V_{1,p}(x_{i_0}-x_{i_1},\ldots,x_{i_0} - x_{i_p})
\end{equation*}
is
\begin{multline}
\label{Hamiltonian_equation_V1p}
\ii \partial_t u + \Delta u = \frac{\mu}{p+1} \left[ \int_{\mathbf{X}^p} \left( \sum_{i=0}^p\prod_{\substack{j = 0 \\ j \neq i}}^p \omega(x_i-x_j)  \right) |u(x_1)|^2\ldots|u(x_p)|^2\,\dd x_1 \ldots \dd x_p \right]u(x) \\
= \frac{\mu}{p+1}\left[ \int_{\mathbf{X}^p} \left(\prod_{\substack{j = 1}}^p \omega(x-x_j)\right) |u(x_1)|^2\ldots|u(x_p)|^2\,\dd x_1 \ldots \dd x_p \right]u(x) \\
+ \frac{\mu \, p}{p+1} \left[ \int_{\mathbf{X}^p} \left( \prod_{\substack{j =0}}^{p-1} \omega(x_j-x_p) \right) |u(x_1)|^2\ldots|u(x_p)|^2\,\dd x_1 \ldots \dd x_p \right]u(x).
\end{multline}
Here we used the evenness of the interaction potential $\omega$. The corresponding energy is given by
\begin{equation}
\label{V1p_classic_Hamiltonian}
E(u) := \int_\mathbf{X} |\nabla u(x)|^2 \dd x
+ \frac{\mu}{p+1} \int_{\mathbf{X}}(\omega * |u|^2)^{p}|u(x)|^2 \, \dd x. 
\end{equation}
We have the following bound on the energy.
\begin{lemma}
Let $p \geq 3$. Suppose $\omega \colon \mathbf{X} \to \R$ is an even $L^1$ function and $u \in H^{s_c}$ for $s_c$ as in \eqref{critical_sobolev}. Then the Hamiltonian defined in \eqref{V1p_classic_Hamiltonian} is finite.
\end{lemma}
\begin{proof}
The first term in \eqref{V1p_classic_Hamiltonian} is bounded as in the proof Lemma \ref{energy_finite_lemma}. The second term is given by
\begin{equation*}
\frac{1}{p+1} \int (\omega * |u|^2)^{p}|u(x)|^2 \, \dd x.
\end{equation*}
This is bounded by applying H\"older's and Young's inequality, followed by the Sobolev embedding theorem. Here we use that for $p \geq 2$,
\begin{equation*}
s_c(p) \geq \frac{3}{2} \left(1 - \frac{1}{p+1}\right),
\end{equation*}
which implies $\|u\|_{L^{2p+2}} \lesssim \|u\|_{H^{s_c}}$
\end{proof}

We can also write down the Hamiltonian structure for \eqref{Hamiltonian_equation_V2p}. Namely, we define a Poisson structure on the space of fields $u \colon \mathbf{X}\to \C$ by the following relation:
\begin{equation}
\label{Poisson_structure}
\{u(x),\overline{u}(y)\} = \ii \delta(x-y), \quad \{u(x),u(y)\} = \{\overline{u}(x),\overline{u}(y)\} = 0.
\end{equation}
We have the following lemma.
\begin{lemma}
\label{Hamiltonian_structure_V1p_lemma}
With the Poisson structure given in \eqref{Poisson_structure}, the Hartree equation given in \eqref{Hamiltonian_equation_V1p} is the Hamiltonian equation of motion associated with the Hamiltonian given in \eqref{V1p_classic_Hamiltonian}.
\end{lemma}
\begin{proof}
Proved by direct computation of $\{H,u\}(x)$ using \eqref{Poisson_structure}. We direct the reader to the proof of \cite[Lemma 1.3]{RS23} for a similar computation.
\end{proof}

Up to a normalisation constant, one can apply the result of Theorem \ref{Convergence_theorem} by directly setting $\omega_n := \omega^*_n$.

\subsubsection*{Generalising $V_2$}

To generalise $V_2$ we consider
\begin{equation}
\label{p_product}
V_{2,p} = \prod_{i \sim j} \omega(x_i-x_j),
\end{equation}
where our product is taken over all pairings of $\{0,1,\ldots,p\}$ with $i < j$, where we write $x \equiv x_0$. The mean-field NLS one expects to be associated to 
\begin{equation*}
H_N = \sum_{j=1}^N -\Delta_j + \frac{\mu}{N^{p}} \sum^N_{1 \leq i_0 < i_1 < \ldots < i_{p} \leq N} V_{2,p}(x_{i_0}-x_{i_1},\ldots,x_{i_0} - x_{i_p})
\end{equation*}
is
\begin{equation}
\label{Hamiltonian_equation_V2p}
\ii \partial_t u + \Delta u = \mu \left[ \int_{\mathbf{X}^p} \prod_{i \sim j} \omega(x_i-x_j) |u(x_1)|^2\ldots|u(x_p)|^2\,\dd x_1 \ldots \dd x_p \right]u(x).
\end{equation}
Here the corresponding energy is given by
\begin{multline}
\label{V2p_classic_Hamiltonian}
E(u) := \int_X |\nabla u(x)|^2 \dd x \\ 
+ \frac{\mu}{p+1} \int_{\mathbf{X}^p} V_{2,p}(x-x_1,\ldots,x-x_p)|u(x_1)|^2 \ldots |u(x_p)|^2 |u(x)|^2 \, \dd x_1 \ldots \dd x_p \, \dd x. 
\end{multline}
We have the following bound on the energy.
\begin{lemma}
Suppose $p \geq 3$. Suppose that $\omega \colon \mathbf{X} \to \R$ is an even $L^{p+1}$ function and $u \in H^{s_c}$ for $s_c$ as in \eqref{critical_sobolev}. Then the Hamiltonian defined in \eqref{V2p_classic_Hamiltonian} is finite.
\end{lemma}
\begin{proof}
We only bound the second term in \eqref{V2p_classic_Hamiltonian}, which we denote $E_p(u)$. We apply H\"older's inequality in the $x_0 \equiv x$ variable. There are $p$ $\omega$'s containing an $x$, and the $|u(x)|^2$. So we obtain a $\|\omega\|_{L^{p+1}}^p$ and a $\|u\|_{L^{2(p+1)}}^2$. We now apply H\"older's inequality in the $x_1$ variable. There are $p-1$ remaining $\omega$'s containing an $x_1$ and the $|u(x)|^2$. We repeat this for all of the $x_i$ variables. We thus get the following bound.
\begin{equation*}
|E_p(u)| \leq \prod_{j=0}^p \|\omega\|_{L^{p+1-j}}^{p-j} \|u\|_{L^{2(p+1-j)}}^2.
\end{equation*}
The assumption on $\omega$ and the Sobolev embedding theorem imply that this is finite.
\end{proof}
\begin{remark}
We remark that $\omega \in L^{p+1}$ is sufficient for the energy to be finite, but is probably not necessary.
Indeed, in the case $p=2$, one can do better; see the discussion after \eqref{3_cluster_energy}.
We do not comment further.
\end{remark}
We can also write down the Hamiltonian structure for \eqref{Hamiltonian_equation_V2p}. 

\begin{lemma}
\label{Hamiltonian_structure_V2p_lemma}
With the Poisson structure given in \eqref{Poisson_structure}, the Hartree equation given in \eqref{Hamiltonian_equation_V2p} is the Hamiltonian equation of motion associated with the Hamiltonian given in \eqref{V2p_classic_Hamiltonian}.
\end{lemma}
\begin{proof}
See the proof of Lemma \ref{Hamiltonian_structure_V1p_lemma}.
\end{proof}

To apply the result of Theorem \ref{Convergence_theorem}, one notices that there are $\frac{p(p+1)}{2}$ terms in the product in \eqref{p_product}. So, arguing as in the previous case, we take $\omega^*_n \to \delta$, and set $\omega_n := (\omega^*_n)^{\frac{2}{p+1}}$.
Here, the exponent comes from $\frac{p}{\frac{p(p+1)}{2}}$ -- i.e. we want the product of $p$ many $\delta$'s, and we have $\frac{p(p+1)}{2}$ interaction potentials in the products. Adding an appropriate normalisation constant, which will depend on the choice of $p$, one can conclude that $V_{2,p,n} \to \delta_p$, and so one can apply the results of Theorem \ref{Convergence_theorem}.

\subsection*{Acknowledgements} This research has been funded by the ANR-DFG project (ANR-22-CE92-0013, DFG PE 3245/3-1 and BA 1477/15-1). A.R. thanks Siegfried Spruck for a helpful discussion concerning the Appendix. The authors thank Zied Ammari for reading this manuscript and for his helpful comments.


\begin{thebibliography}{References}

\bibitem{ARR22} A.K.~Arora, O.~Riaño, S.~Roudenko, \emph{Well-posedness in weighted spaces for the generalized Hartree equation with $p<2$}, Commun.
Contemp. Math. 24 (2022), no. 09, 2150074.

\bibitem{AR20} A.K.~Arora, S.~Roudenko, \emph{Well-posedness and blow-up properties for the generalized Hartree equation},
Journal of Hyperbollic Differential Equations, Volume 17, No.~4 (2020).

\bibitem{BPS16} N.~Benedikter, M.~Porta, B.~Schlein, \emph{Effective Evolution Equations from Quantum Dynamics}, Springer Briefs in Mathematical Physics, Springer, Berlin (2016).

\bibitem{BS23} D.G.~Bhimani, S.~Haque, \emph{Strong ill-posedness for fractional Hartree and cubic NLS equations}, Journal of Functional Analysis, Volume 285, Issue 11 (2023).

\bibitem{Bou93} J.~Bourgain, \emph{Fourier transform restriction phenomena for certain lattice subsets and applications to nonlinear evolution equations. I. Schr\"odinger equations}. Geom. Funct. Anal., Volume 3 (1993), 107--156.

\bibitem{BD15} J.~Bourgain, C.~Demeter, \emph{The proof of the $l^2$ decoupling conjecture}. Ann. of Math. (2) 182, 1 (2015), 351--389.

\bibitem{CW90} T. Cazenave, F. Weissler,
\emph{The cauchy problem for the critical nonlinear Schrödinger equation in $H^s$},
Nonlinear Analysis: Theory, Methods \& Applications,
Volume 14, Issue 10, 807--836, (1990).

\bibitem{CBU24} T.~Chen, A. Bowles Urban, \emph{On the Well-Posedness and Stability of Cubic and Quintic Nonlinear Schr\"odinger Systems on 
$\T^3$}. Ann. Henri Poincar\'e 25, 1657--1692 (2024).

\bibitem{CP10} T.~Chen, N.~Pavlovi\'c, \emph{On the Cauchy problem for focusing and defocusing Gross-Pitaevskii hierarchies}, Discrete Contin. Dyn. Syst. 27 (2010), no. 2, 715–739.

\bibitem{CP11} T.~Chen, N.~Pavlovi\'c, \emph{The quintic NLS as the mean field limit of a boson gas with three-body interactions}, J. Funct. Anal., 260(4):959–-997, (2011).

\bibitem{CH19} X. Chen, J. Holmer, \emph{The derivation of the energy-critical NLS from quantum many-body dynamics}, Invent. math. (2019) {\bf 217}, 433–547.

\bibitem{CKSTT08} J. Colliander, M. Keel, G. Staffilani, H. Takaoka, T. Tao, \emph{Global well-posedness and scattering for the energy-critical nonlinear Schr\"odinger equation in $\mathbb{R}^3$}, Ann. of Math. (2) 167 (2008), no. 3, 767-865.

\bibitem{FKSS18} J. Fr\"ohlich, A. Knowles, B. Schlein, V. Sohinger, \emph{A microscopic derivation of time-
dependent correlation functions of the 1D cubic nonlinear Schr\"odinger equation}, Adv. Math. 353 (2019), 67--115.

\bibitem{GV80} J.~Ginibre, G.~Velo, \emph{On a class of non linear Schrödinger equations with non local interaction}, Math Z 170, 109–136 (1980).

\bibitem{GOW14} Z.~Guo, T.~Oh, Y.~Wang, \emph{Strichartz estimates for Schr\"odinger equations on irrational tori}. Proceedings of the London Mathematical Society 109.4 (2014), 975-1013.

\bibitem {GX23} C. Guzm\'an, C. Xu, \emph{The energy-critical inhomogeneous generalized Hartree equation in 3D}, Preprint arXiv, 2305.00972 (2023). 

\bibitem{HHK09} M. Hadac, S. Herr, H. Koch, \emph{Well-posedness and scattering for the KP-II equation in a critical space}, Ann. Inst. H. Poincaré Anal. Non Lin\'eaire 26 (2009), no. 3, pp. 917–941.

\bibitem{Hep74} K. Hepp, \emph{The classical limit for quantum mechanical correlation functions}, Commun.Math. Phys. {\bf 35}, 265–277 (1974).

\bibitem{HTT11} S. Herr, D. Tataru, N. Tzvetkov, \emph{Global well-posedness of the energy-critical nonlinear Schr\"odinger equation
with small initial data in $H^1(\T^3)$}. Duke Math. J. 159, 2 (2011), 329–349.

\bibitem{HTX16} Y. Hong, K. Taliaferro, Z. Xie,
\emph{Uniqueness of solutions to the 3D quintic Gross–Pitaevskii hierarchy}, Journal of Functional Analysis, Volume 270, Issue 1, 34--67, (2016).

\bibitem{IP12b} A.D. Ionescu, B. Pausader, \emph{Global well-posedness of the energy-critical defocusing NLS on $\R \times \T^3$}. Commun. Math. Phys. 312 (2012), 781–831.

\bibitem{IP11} A.D. Ionescu, B. Pausauder, \emph{The energy-critical defocusing NLS on $\T^3$}, Duke Math. J., Volume 161 (2012), 1581–1612.

\bibitem{JV24} L. Junge, F.L.A. Visconti, \emph{Ground state energy of a dilute Bose gas with three-body hard-core interactions}, Preprint arXiv 2406.09019 (2024).

\bibitem{Kat87} T. Kato, \emph{On nonlinear Schr\"odinger equations}, Ann. Inst. H. Poincare Physique Theorique 46, 113-129 (1987). 

\bibitem{KV16} R. Killip and M. Vi\c{s}an, \emph{Scale invariant Strichartz estimates on tori and applications}, Math. Res. Lett., Volume 23 (2016), 445–472.

\bibitem{Lee19} G.E. Lee, \emph{Local wellposedness for the critical nonlinear Schr\"odinger equation on $\T^3$}. Discrete Contin. Dyn.
Syst. 39 (2019), no. 5, 2763–2783.

\bibitem{Lee20} J. Lee, \emph{Rate of convergence towards mean-field evolution for weakly interacting bosons with singular three-body interactions}, Preprint arXiv 2006.13040 (2020).

\bibitem{LYYZ23} Y. Luo, X. Yu, H. Yue, Z. Zhao, \emph{On well-posedness results for the cubic-quintic NLS on $\T^3$}. Preprint arXiv 2301.13433 (2023).

\bibitem{NRT23} P.T. Nam, J. Ricaud, A. Triay, \emph{The condensation of a trapped dilute Bose gas with three-body interactions}. Probability and Mathematical Physics, Vol. 4 (2023), No. 1, 91–149.

\bibitem{NS20} P.T. Nam, R. Salzmann, \emph{Derivation of 3D energy-critical nonlinear Schr\"odinger equation and Bogoliubov excitations for Bose gases}. Comm. Math.Phys., 375(1):495–571, 2020.

\bibitem{RS23} A. Rout, V. Sohinger, \emph{A microscopic derivation of Gibbs measures for the 1D focusing quintic nonlinear Schr\"odinger equation}. Preprint arXiv 2308.06569 (2023).

\bibitem{Spo80} H. Spohn, \emph{Kinetic Equations from Hamiltonian Dynamics}, Rev. Mod. Phys, {52\bf} (1980), no. 3, 569–615.

\bibitem{TVZ07} T. Tao, M. Vi\c{s}an, X. Zhang, \emph{The Nonlinear Schrödinger Equation with Combined Power-Type Nonlinearities}. Communications in Partial Differential Equations, 32(8), 1281–1343 (2007).

\bibitem{Wang13} Y. Wang, \emph{Periodic nonlinear Schrödinger equation in critical $H^s(\T^n)$ spaces}, SIAM J. Math. Anal. 45 (2013), no. 3, 1691–1703.

\bibitem{Xie15} Z. Xie, \emph{Derivation of a nonlinear Schr\"odinger equation with a general power-type nonlinearity in $d = 1, 2$}, Differential Integral Equations 28 (2015), no. 5-6, 455–504.

\bibitem{Zhang06} X. Zhang, \emph{On the Cauchy problem of 3-D energy-critical Schrödinger equations with subcritical perturbations}, Journal of Differential Equations 230 (2006): 422-445.
\end{thebibliography}
\end{document}